\documentclass[10pt]{oxarticle}

\usepackage{graphicx, float, array, xspace, amscd, amsmath, amsthm, amssymb, latexsym, mathtools, bbm}
\usepackage[sort&compress, comma, square, numbers]{natbib}
\usepackage{multicol}
\usepackage{tikz}
\usepackage{bbm}
\usetikzlibrary{calc,arrows}
\usepackage{html}
\usepackage[T1]{fontenc}
\usepackage{lmodern}

\usepackage{dsfont}

\newlength{\xtrawidth}
\newlength{\xtraheight}
\DeclareFontFamily{OT1}{pzc}{}
\DeclareFontShape{OT1}{pzc}{m}{it}{<-> s * [1.200] pzcmi7t}{}
\DeclareMathAlphabet{\mathpzc}{OT1}{pzc}{m}{it}

\newcommand{\cK}{\mathcal{K}}

\newcommand{\cO}{\mathcal{O}}

\newcommand{\IC}{\mathbb{C}}

\newcommand{\IP}{\mathbb{P}}

\newcommand{\IR}{\mathbb{R}}

\newcommand{\IZ}{\mathbb{Z}}

\font\eightrm=cmr8 at 8pt
\font\csc=cmcsc10

\DeclareFontFamily{U}{wncy}{}
\DeclareFontShape{U}{wncy}{m}{n}{<->wncyr10}{}
\DeclareSymbolFont{mcy}{U}{wncy}{m}{n}
\DeclareMathSymbol{\sha}{\mathord}{mcy}{"58}

\newcommand{\place}[3]{\vbox to0pt{\kern-\parskip\kern-7pt
                             \kern-#2truein\hbox{\kern#1truein #3}
                             \vss}\nointerlineskip}

\newcommand{\beq}{\begin{equation}}
\newcommand{\eeq}{\end{equation}}
\newcommand{\beqnn}{\begin{equation*}}
\newcommand{\eeqnn}{\end{equation*}}
\newcommand{\fref}[1]{Figure~\ref{#1}}

\newcommand{\+}{\hphantom{-}}

\newcommand{\cicy}[2]{\begin{matrix} #1\end{matrix}\!\left[\begin{matrix}#2 \end{matrix}\right]}

\newcommand{\quotient}[1]{_{\hskip-2pt\lower1pt\hbox{$/$}\lower2pt\hbox{\hskip-1pt$#1$}}}

\numberwithin{equation}{section}
\proofmodetrue

\newtheorem{theorem}{Theorem}

\newcommand{\nnbe}{\begin{equation*}}
\newcommand{\nnee}{\end{equation*}}
\newcommand{\bea}{\begin{eqnarray}}
\newcommand{\eea}{\end{eqnarray}}
\newcommand{\ba}{\begin{align}}
\newcommand{\ea}{\end{align}}
\newcommand{\bi}{\begin{itemize}}
\newcommand{\ei}{\end{itemize}}
\newtheorem{thm}{Theorem}[section]
\theoremstyle{definition}
\newtheorem{defn}[thm]{Definition}
\theoremstyle{plain}
\newtheorem{proposition}[thm]{Proposition}
\newtheorem{conjecture}[theorem]{Conjecture}
\newtheorem{con}[thm]{Conjecture}
\newtheorem{crl}[thm]{Corollary}

\theoremstyle{definition}

\newtheorem{remark}[thm]{Remark}
\theoremstyle{definition}
\newtheorem{example}[thm]{Example}
\newtheorem*{nndefn}{Definition}
\theoremstyle{plain}

\newsavebox{\overlongequation}

\newcommand{\fl}[1]{{#1}'}					
\newcommand{\contr}[1]{{#1}_\mathrm{sing}}

\newcommand{\xdashrightarrow}[2][]{\ext@arrow 0359\rightarrowfill@@{#1}{#2}}

\newcommand{\xdasharrow}[2][->]{
\tikz[baseline=-\the\dimexpr\fontdimen22\textfont2\relax]{
\node[anchor=south,font=\scriptsize, inner ysep=1.5pt,outer xsep=2.2pt](x){#2};
\draw[shorten <=3.4pt,shorten >=3.4pt,dashed,#1](x.south west)--(x.south east);
}
}

\newcommand{\ceil}[1]{\left\lceil{#1}\right\rceil}

\newcommand{\PP}{\mathbb{P}}
\newcommand{\RR}{\mathbb{R}}

\def\+{\oplus}                   
\def\*{\otimes}                  

\def\mov{\operatorname{Mov}}

\def\eff{\operatorname{Eff}}
\def\nef{\operatorname{Nef}}

\proofmodefalse

\begin{document}
\pagestyle{empty}
\begin{center}
\null\vskip0.3in
{\LARGE \textsc{Generating Functions  \\[3pt]
for Line Bundle Cohomology Dimensions \\[3pt] on Complex Projective Varieties}\\[0.59in]}
{\csc Andrei Constantin\\[1.3cm]}
{\it 
Rudolf Peierls Centre for Theoretical Physics\\
University of Oxford\\ 
Parks Road, Oxford OX1 3PU, UK}
\footnotetext[1]{andrei.constantin@physics.ox.ac.uk}

\vfill
{\bf Abstract\\[2ex]}
\parbox{6.0in}{\setlength{\baselineskip}{14pt}
This paper explores the possibility of constructing multivariate generating functions for all cohomology dimensions of all holomorphic line bundles on certain complex projective varieties of Fano, Calabi-Yau and general type in various dimensions and Picard numbers. Most of the results are conjectural and rely on explicit cohomology computations. We first propose a generating function for the Euler characteristic of all holomorphic line bundles on complete intersections in products of projective spaces and toric varieties. This generating function is constructed by expanding the Hilbert-Poincar\'e series associated with the coordinate ring of the variety around all possible combinations of zero and infinity and then summing up the resulting contributions with alternating signs. Similar generating functions are proposed for the individual cohomology dimensions of all holomorphic line bundles on certain complete intersections, including examples of Mori and non-Mori dream spaces. Surprisingly, the examples studied indicate that a single generating function encodes both the zeroth and all higher cohomologies upon expansion around different combinations of zero and infinity, raising the question whether such generating functions determine the variety uniquely. 
}

\end{center}
\newpage
\begingroup
\baselineskip=14pt
\tableofcontents
\endgroup
\newpage
\setcounter{page}{1}
\pagestyle{plain}
\section{Introduction and Overview}

Recent studies of line bundle cohomology data on certain classes of complex projective varieties have revealed the existence of relatively simple cohomology functions on spaces where one would not normally expect to find them. Initially, such functions were found empirically -- both for the zeroth and the higher cohomologies -- through a combination of direct observation~\cite{Constantin:2018hvl, Larfors:2019sie, Brodie:2019pnz} and machine learning~\cite{Klaewer:2018sfl, Brodie:2019dfx}, followed by a series of definitive results in dimensions two and three concerning the exact form of the zeroth line bundle cohomology function~\cite{Brodie:2019ozt, Brodie:2020wkd, Brodie:2020fiq}. 

The prospect of expressing line bundle cohomology dimensions in closed form, especially in the case of Calabi-Yau varieties, is of particular interest in string theory, where bundle-valued cohomology plays a crucial role in the computation of particle spectra; indeed, this was the primary motivation behind the studies cited above, as~reviewed, for instance, in Refs.~\cite{Brodie:2021zqq, Constantin:2022jyd}. While in straightforward cases it is feasible to compute line bundle cohomology directly from its definition, obtaining this information for more involved examples such as those arising in string theory compactifications can be challenging, if at all possible. As such, obtaining closed form expressions for line bundle cohomology dimensions can significantly simplify the analysis of string theory models~\cite{Abel:2023zwg, Constantin:2021for}, providing strong motivation for a more thorough investigation. From a mathematical point of view, sheaf cohomology and, in particular, line bundle cohomology represents a powerful tool in algebraic geometry and understanding the extent to which the closed form cohomology functions identified so far can be further generalised is an important task. Despite the aspiration for a broad generalisation, the present paper remains focused on specific examples. However, the paper brings a shift of perspective, from closed form cohomology functions to cohomology generating functions. Moreover, the examples discussed here include varieties of all types, Fano, Calabi-Yau and general type in arbitrary dimension. 

The current status of closed form line bundle cohomology functions can be briefly summarised as follows. In~dimension two, for toric surfaces, weak Fano surfaces and K3 surfaces, it was shown that the zeroth cohomology function arises through an interplay between Zariski decomposition and certain vanishing theorems, leading to a decomposition of the effective cone into a number of polyhedral  sub-cones, inside which the zeroth cohomology function can be equated to a topological index~\cite{Brodie:2019ozt, Brodie:2020wkd}. In fact, in all these cases the cohomology chambers correspond either to the nef cone or to Zariski chambers~\cite{Brodie:2020wkd, Bauer04}. The information required to specify the zeroth cohomology function and the associated chamber structure consists of the Mori cone generators (assuming their number is finite) and the intersection form on the surface. The higher cohomology functions are then inferred by Serre duality and the Riemann-Roch theorem.

In dimension three, the studies undertaken in Refs.~\cite{Brodie:2020fiq, Brodie:2021nit} focused mainly on Calabi-Yau varieties, indicating a similar chamber decomposition structure of the effective cone. Although no general statements have been proven, in the examples studied so far the zeroth cohomology function was understood to arise through a combination of a three-dimensional version of Zariski decomposition, the Kawamata-Viehweg vanishing theorem and the invariance of the dimension of linear systems of divisors under flops. As such, the information required to specify the zeroth cohomology function, except possibly along certain boundaries of the effective cone which needed a separate treatment, consisted of the collection of all flops connecting the variety to its birational models, complemented by knowledge of the set of all rigid divisors on the variety. Much less could be said about the structure of the first and second cohomology functions, although the empirical results of Refs.~\cite{Brodie:2020fiq, Brodie:2021nit} suggest that these can also be expressed in closed form. 

In the present paper we ask the natural question whether line bundle cohomology data can be encoded into exact generating functions. An indication that this may indeed be the case comes from the well-known fact that the Hilbert-Poincar\'e series of a projective variety represents a generating function for the dimensions of the graded pieces of its homogeneous coordinate ring. Since very positive line bundles have all global sections given by polynomials, the Hilbert-Poincar\'e series encodes the information about the dimensions of their zeroth cohomologies. However, in general there are effective line bundles that have non-polynomial sections.

To~remediate this, one needs to consider the Hilbert-Poincar\'e series associated with the Cox ring of the variety which, by definition, encodes the zeroth cohomology dimensions of all line bundles~\cite{Hu:2000}. 
Several of the examples presented in this paper follow this approach for constructing a generating function for the zeroth cohomology. On the other hand, writing down an explicit presentation for the Cox ring can be difficult in general, requiring a lot of information about linear systems and divisors on the variety. As discussed below, an alternative method for constructing the zeroth line bundle cohomology generating function is to exploit directly the birational structure of the variety, an approach which in particular allows us to treat examples where the Cox ring is not finitely generated (non-Mori dream spaces). For the higher cohomologies there is no obvious candidate, however, the examples discussed here suggest that the same generating function that encodes the zeroth line bundle cohomology, also encodes the higher cohomologies upon an appropriate choice of the points of expansion.  

Most of the results in this paper are conjectural and rely on explicit cohomology calculations for a sufficiently large range of line bundles, performed with various computer packages (mainly~\cite{cicypackage}, also~\cite{cicytoolkit, cohomCalg:Implementation}). In the remaining part of this section we give an overview of the main results. 

{\bfseries The Euler characteristic.}
We start with a generating function for the Euler characteristic of all line bundles over complete intersections in products of projective spaces or toric varieties. 

\begin{conjecture}\label{con:Euler_intro}
Let $X\hookrightarrow \IP^{n_1}\times \IP^{n_2}\times\ldots\times \IP^{n_p}$ be a complete intersection in a product of $p$ projective spaces, cut out by $q$ homogeneous polynomials of multi-multidegrees $d^{(i)} = (d_{1,i}, d_{2,i},\ldots,d_{p,i})$ for $1\leq i\leq q$. Let $H_i=\left( \cO_{\IP^{n_1}}(0)\otimes\ldots\otimes \cO_{\IP^{n_i}}(1) \otimes\ldots\otimes  \cO_{\IP^{n_p}}(0)\right)|_X$ and
\begin{equation}
\begin{aligned}
HS(X,t_1,t_2,\ldots, t_p) &= \frac{\prod_{j=1}^q \left(1- \prod_{i=1}^p t_i^{d_{i,j}} \right)  }{\prod_{i=1}^p \left(1- t_i \right)^{n_i+1}} ~
\end{aligned}
\end{equation}
be the multivariate Hilbert-Poincar\'e series of $X$, where the expansion is around $t_i=0$.
A generating function for the Euler characteristic of all line bundles on $X$ obtained as restrictions from the ambient variety is given by 
\begin{equation}\label{eq:GFeuler}
\sum_{\sigma\in S} (-1)^{{\rm sign}(\sigma)} HS\left(X,\begin{array}{ccc} t_1& \ldots& t_p\\\sigma(1)&\ldots& \sigma(p) \end{array}\right)  = \sum_{m_i\in \mathbb Z} \chi(X, H_1^{\otimes m_1}\otimes\ldots \otimes H_p^{\otimes m_p})\, t_1^{m_1}\ldots t_p^{m_p}
\end{equation}
where the order of expansion for the Hilbert-Poincar\'e series is indicated by the order in which the variables appear as arguments and the point of expansion is indicated below each variable, namely $t_i=\sigma(i)$. Moreover, $S=\{0,\infty\}^{\times p}$, and ${\rm sign}(\sigma)$ is equal to the number of times $\infty$ appears in $\sigma$.
\end{conjecture}

The conjecture is discussed in Section~\ref{sec:EulerChar}. The fact that the Hilbert-Poincar\'e series eventually agrees with the Euler characteristic for large enough and positive exponents $m_i$ is well-known. However, the proposed generating function covers all line bundles on $X$ that can be obtained as restrictions of line bundles from the ambient variety. Embeddings for which the entire Picard group of $X$ descends from the ambient variety are called favourable and in such cases the generating function~\eqref{eq:GFeuler} covers the entire Picard group. Note that the conjecture makes no assumption on the type of the variety (Fano, Calabi-Yau or general), nor on its dimension. The conjecture can easily be extended to complete intersections in toric varieties as discussed in Section~\ref{sec:EulerChar}. 

For individual cohomology dimensions, the results proposed here refer to specific examples or classes of examples. 

\begin{nndefn}
Given a complex projective variety $X$ with Picard number $r$, we define the {\itshape cohomology series} 
\begin{equation}
CS^i(X,\cO_X) = \sum_{m_i\in \mathbb Z}  h^i(X,\cO_{X}(m_1H_1+\ldots m_r H_r))\, t_1^{m_1}\ldots t_r^{m_r}~,
\end{equation}
for $0\leq i\leq {\rm dim}(X)$, relative to a $\mathbb Z$-basis $\{H_1,H_2,\ldots H_r\}$ of $\text{Pic}(X)$.
\end{nndefn}

{\bfseries A simple example in Picard number one.}
For simplicity, the examples discussed in this paper have low Picard number (up to $4$). To start with a very simple class of varieties in Picard number one, which includes the quintic Calabi-Yau threefold, consider a smooth generic hypersurface $X\subset \IP^{n\geq 4}$ of degree $d$. For $n \geq 4$, the Lefschetz hyperplane theorem ensures that Pic$(X)\cong \mathbb{Z}$. Let $H=\cO_{\IP^n}(1)|_X$. The Hilbert-Poincar\'e series associated with the coordinate ring of $X$ encodes the zeroth cohomology dimensions for all line bundles on $X$:
\begin{equation}\label{eq:coh0Pnd_intro}
HS(X=\IP^n[d],t) = \frac{1-t^d}{(1-t)^{n+1}} = \sum_{m\in \mathbb Z}  h^0(X,\cO_{X}(mH))\, t^m~
\end{equation}
and we write this as
\begin{equation}
CS^0(X,\cO_X) = \left( \frac{1-t^d}{(1-t)^{n+1}}~, \begin{array}{c} t\\0 \end{array} \right) = \sum_{m\in \mathbb Z}  h^0(X,\cO_{X}(mH))\, t^m~,
\end{equation}
where $CS^0(X,\cO_X)$ stands for the zeroth cohomology series associated with the structure sheaf of~$X$ and the expansion is taken around $t=0$.
Eq.~\eqref{eq:coh0Pnd_intro} follows from the additivity of the Hilbert-Poincar\'e series relative to the Koszul resolution
\begin{equation}
0\rightarrow \cO_{\IP^n}(-d)\rightarrow \cO_{\IP^n} \rightarrow \cO_X\rightarrow 0~,
\end{equation}
from which 
$t^d HS(\IP^n,t) - HS(\IP^n,t) + HS(\IP^n[d],t) =0$. The same sequence implies that the only other non-vanishing line bundle cohomology is the top one, since  $h^i(\IP^n, L)=0$ for all $0<i<n$ and any line bundle $L$. Moreover, the information about the top line bundle cohomology is also encoded in the Hilbert-Poincar\'e series expanded around infinity, a statement which follows from the Serre duality relation 
\begin{equation}
h^0(X,\cO_X(mH))=h^{n-1} (X,\cO_X(-mH)\otimes\cO_X((d{-}n{-}1)H))~.
\end{equation} 
Formally, obtaining the top cohomology series involves replacing $t\rightarrow t^{-1}$ in $HS(X,t)$, followed by the expansion around $\{t=0\}$ and a second replacement $t\rightarrow t^{-1}$:
\begin{equation}
CS^{n-1}(X,\cO_X) = \left( \frac{1-t^d}{(1-t)^{n+1}}~, \begin{array}{c} t\\\infty \end{array} \right) = \sum_{m\in \mathbb Z}  h^{n-1}(X,\cO_{X}(mH))\, t^m~.
\end{equation}

While this example is straightforward, it reveals a recurring pattern that runs throughout the paper, namely the existence of a `universal' generating function (here the Hilbert-Poincar\'e series itself) encoding the dimensions of all cohomologies of all line bundles in the Picard group. The examples discussed in the rest of the paper display a variety of contexts in which such `universal' generating functions have been identified. 

{\bfseries Hirzebruch surfaces.}
In dimension two, we prove the existence of a universal generating function for the cohomology of line bundles on Hirzebruch surfaces using the Cech cohomology tools developed in Refs.~\cite{CohomOfLineBundles:Algorithm, CohomOfLineBundles:Proof, schenck2010euler, Jow:2011, Blumenhagen:2011xn}, as discussed in Section~\ref{sec:Hirzebruch}. Unfortunately, the proof does not immediately generalise to other cases (beyond toric examples), as tracking Cech cohomology representatives through spectral sequences is difficult.

\begin{theorem}\label{thm:HirzebruchS}
Let $\mathbb F_n$ be the $n$-th Hirzebruch surface $\IP(\cO_{\IP^1}\oplus \cO_{\IP^1}(n))$
with Picard group  ${\rm Pic}({\mathbb F_n}) = \mathbb Z C\oplus \mathbb Z F$, where $C$ is the unique irreducible curve with negative self intersection $C^2=-n$ and $F$ corresponds to the fiber with $F^2=0$, $F\cdot C = 1$. The nef cone is ${\rm Nef}({\mathbb F_n})=\IR_{\geq 0} F +\IR_{\geq 0} (nF+C)$ and the effective cone is ${\rm Eff}({\mathbb F_n})=\IR_{\geq 0} F +\IR_{\geq 0} C~.$
The cohomology dimensions of all line bundles in ${\rm Pic}({\mathbb F_n})$ are encoded by the following cohomology series relative to the basis $\{C,F\}$:
\begin{equation}
\begin{aligned}
CS^0(\mathbb F_n,\mathcal O_{\mathbb F_n})&{=} \left(\frac{1}{(1-t_1 t_2^n)(1-t_2)^2(1-t_1)},\begin{array}{cc} t_1& t_2 \\0&0 \end{array}\right)\\
CS^1(\mathbb F_n,\mathcal O_{\mathbb F_n})&{=} \left(\frac{1}{(1-t_1 t_2^n)(1-t_2)^2(1-t_1)},\begin{array}{cc} t_1& t_2 \\ \infty&0 \end{array}\right) +\left(\frac{1}{(1-t_1 t_2^n)(1-t_2)^2(1-t_1)},\begin{array}{cc} t_1& t_2 \\0& \infty\end{array}\right)\\
CS^2(\mathbb F_n,\mathcal O_{\mathbb F_n})&{=} \left(\frac{1}{(1-t_1 t_2^n)(1-t_2)^2(1-t_1)},\begin{array}{cc} t_1& t_2 \\\infty&\infty \end{array}\right)  
\end{aligned}
\end{equation}
\end{theorem}

Note that the middle cohomology generating function has two contributions, since the first cohomology is non-vanishing within the two disconnected cones $-2C+\IR_{\geq -1} F +\IR_{\leq 0}(C+F)$ and $\IR_{\leq -2}F + \IR_{\geq 0}(C+F)$.  

{\bfseries Hypersurfaces in $\IP^1\times \IP^n$.}
Moving up in dimension, we propose the following. 

\begin{conjecture}\label{con:coh_intro1}
Let $X$ be a general hypersurface of bi-degree $(d,e)$ in $\IP^1\times \IP^{n\geq 3}$ with $d\leq n$ and $e$ arbitrary or $d$ arbitrary and $e=1$. Denote $H_1 = \mathcal O_{\IP^1\times \IP^n}(1,0)|_X$ and $H_2 = \mathcal O_{\IP^1\times \IP^n}(0,1)|_X$. Then in the basis $\{H_1, H_2\}$,
\begin{equation}\label{eq:hypersurfaces_intro}
\begin{aligned}
CS^0(X,\mathcal O_X)&{=}\left( \frac{(1{-}t_2^{e})^{d+1}}{(1{-}t_1)^2(1{-}t_2)^{n+1}(1{-}t_1^{-1}t_2^{e})^d}\,,\!\begin{array}{cc} t_2& t_1 \\0&0 \end{array}\!\!\right) {=}\!\!\!\! \sum_{m_1,m_2\in \mathbb Z} \!\! h^0(X,\cO_X(m_1 H_1{+}m_2 H_2)) t_1^{m_1}t_2^{m_2}\\
CS^1(X,\mathcal O_X)&{=}\left( \frac{(1{-}t_2^{e})^{d+1}}{(1{-}t_1)^2(1{-}t_2)^{n+1}(1{-}t_1^{-1}t_2^{e})^d}\,,\!\begin{array}{cc} t_2& t_1 \\0&\infty \end{array}\!\!\right) {=}\!\!\!\! \sum_{m_1,m_2\in \mathbb Z} \!\! h^1(X,\cO_X(m_1 H_1{+}m_2 H_2)) t_1^{m_1}t_2^{m_2}\\
({-}1)^nCS^{n{-}1}(X{,}\mathcal O_X)&{=}\left( \frac{(1{-}t_2^{e})^{d{+}1}}{(1{-}t_1)^2(1{-}t_2)^{n+1}(1{-}t_1^{-1}t_2^{e})^d}\,,\!\begin{array}{cc} t_2& t_1 \\\infty&0 \end{array}\!\!\right) {=}\!\!\!\! \sum_{m_1,m_2\in \mathbb Z} \!\! h^{n-1}(X,\cO_X(m_1 H_1{+}m_2 H_2)) t_1^{m_1}t_2^{m_2}\\
(-1)^nCS^n(X,\mathcal O_X)&{=}\left( \frac{(1{-}t_2^{e})^{d+1}}{(1{-}t_1)^2(1{-}t_2)^{n+1}(1{-}t_1^{-1}t_2^{e})^d}\,,\!\begin{array}{cc} t_2& t_1 \\ \infty&\infty \end{array}\!\!\right) {=}\!\!\!\! \sum_{m_1,m_2\in \mathbb Z} \!\! h^n(X,\cO_X(m_1 H_1{+}m_2 H_2)) t_1^{m_1}t_2^{m_2}\\
\end{aligned}
\end{equation}
and all intermediate line bundle cohomologies vanish. 
\end{conjecture}

The conjecture is discussed in Section~\ref{hypersurfacesP1Pn} together with a similar statement in dimension two which includes 
general hypersurfaces of bidegree $(2,e\geq 3)$ in $\IP^1\times \IP^2$, see Section~\ref{hypersurfacesP1P2}. For the bi-degrees considered in Conjecture~\ref{con:coh_intro1}, the rank of the Picard group is $2$ and the generating functions cover the entire Picard group. The generating function for the zeroth line bundle cohomology dimensions is constructed as the Hilbert-Poincar\'e series associated with the Cox ring of $X$, an explicit presentation of which was given by Ottem in Ref.~\cite{Ottem2015}. The generating functions for the higher cohomologies are conjectural. 
Note that this class of hypersurfaces includes varieties of Fano, Calabi-Yau and general type of arbitrary dimension greater or equal to~two. 

{\bfseries Relation to the birational structure.}
Without making any general statements, we consider in more detail the case $d=2, e=4$ and $n=3$ which corresponds to a general Calabi-Yau hypersurface $X$ in $\IP^1\times \IP^3$, as discussed in Section~\ref{sec:7887general}. An alternative method for constructing the generating function~\eqref{eq:hypersurfaces_intro} exploits directly the birational structure of the variety. Concretely, the zeroth cohomology generating function 
\begin{equation}\label{cs07887_intro}
CS^0(X,\mathcal O_X){=}\left( \frac{(1{-}t_2^{4})^{3}}{(1{-}t_1)^2(1{-}t_2)^{4}(1{-}t_1^{-1}t_2^{4})^2}\,,\!\begin{array}{cc} t_2& t_1 \\0&0 \end{array}\!\!\right) 
\end{equation}
can be decomposed into two contributions associated with the two Mori chambers of~$X$ and an additional correction term associated with the wall separating the two Mori chambers. A general hypersurface of bi-degree $(2,4)$ in $\IP^1\times \IP^3$ admits two birational models related by a flop, which happen to be diffeomorphic to each other. The effective cone is shown in \fref{fig:X7887generic_intro} and is equal to the movable cone, consisting of the union of the nef cones associated with the two diffeomorphic birational models: 
\begin{equation}
{\rm Eff}(X)=\left( \IR_{\ge 0}H_1+\IR_{\ge 0}H_2\right) \cup \left( \IR_{\ge 0}H_2+\IR_{\ge 0}(4H_2-H_1)\right)~.
\end{equation}
On the other hand, the region where the first line bundle cohomology is non-vanishing, also shown in \fref{fig:X7887generic_intro}, corresponds to the cone $-2H_1 + \IR_{\geq 0} (-H_1) + \IR_{\geq 0} H_2$.

\begin{center}
\begin{figure}[h]
\begin{center}
\includegraphics[width=7.2cm]{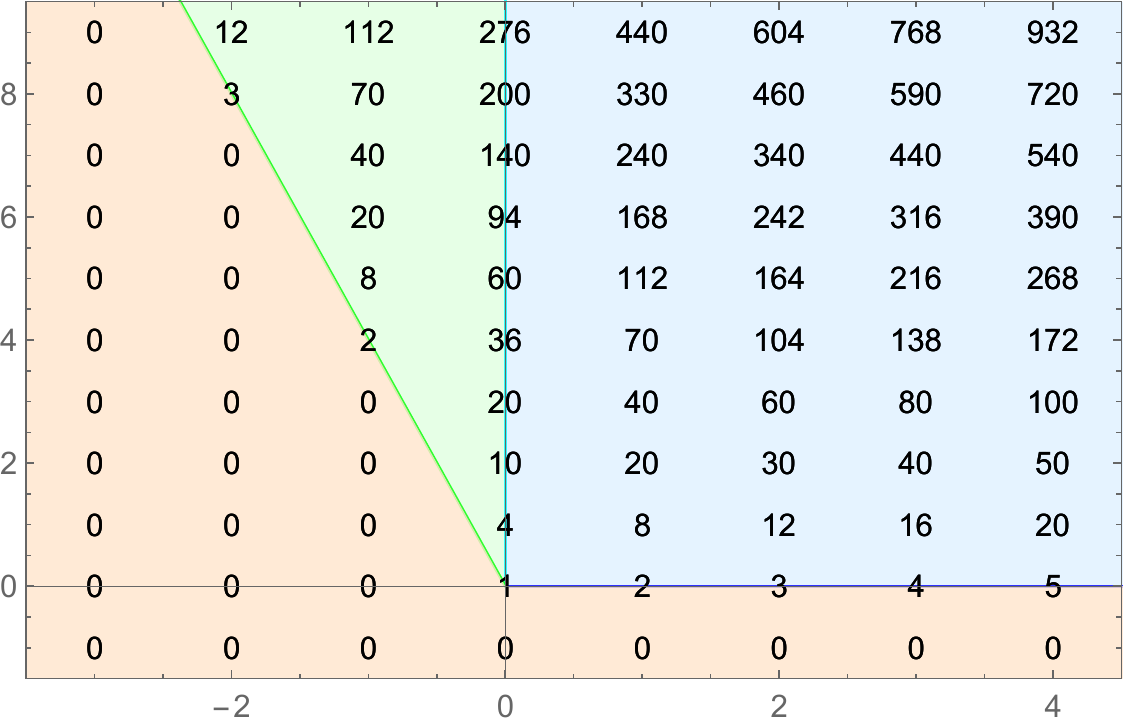}~~~~~
\includegraphics[width=7.2cm]{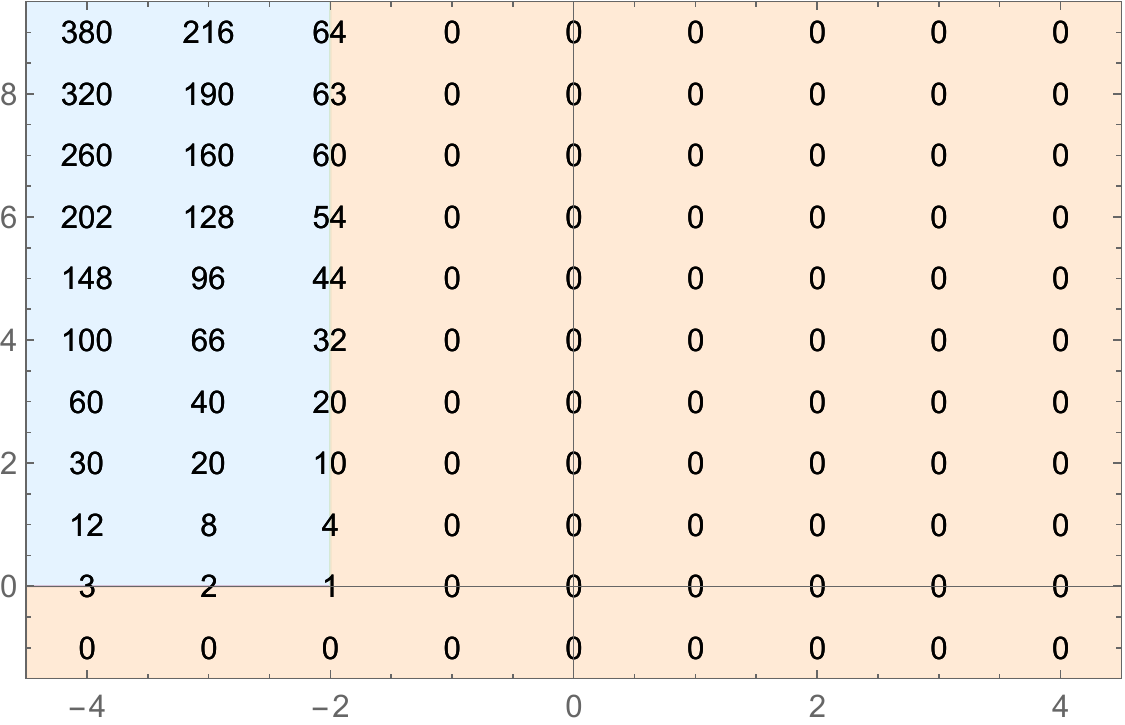}
\caption{\itshape Zeroth and first line bundle cohomology data (left and, respectively, right plot) for a generic hypersurface of bi-degree $(2,4)$ in $\mathbb{P}^1{\times}\mathbb{P}^3$. The numbers indicate cohomology dimensions, while their locations in the plot indicate the first Chern class of the corresponding line bundles.}
\label{fig:X7887generic_intro} 
\end{center}
\end{figure}
\end{center}
The generating function \eqref{cs07887_intro} can then be obtained by adding up the Hilbert-Poincar\'e series associated with the coordinate rings of the two birational models of $X$ and subtracting a correction term: 
\begin{equation}\label{eq:cs07887v2}
\begin{aligned}
CS^0\left(X,\mathcal O_X \right)&= HS(X, t_1, t_2) + HS(X, t_1^{-1}t_2^4, t_2) - {\rm corr.~term}\\
& \!\!\!\!\!\!\!\!\!\!\!\!\!\!\!\!= \left( \frac{1-t_1^2 t_2^4}{(1-t_1)^2 (1-t_2)^4} ,\begin{array}{cc} t_2& t_1 \\0&0 \end{array}\right) + \left( \frac{1-\left(t_1^{-1}t_2^4\right)^2 t_2^4}{(1-t_1^{-1}t_2^4)^2 (1-t_2)^4} ,\begin{array}{cc} t_2& t_1 \\0&0 \end{array}\right) - \left( \frac{1 - t_2^4}{(1-t_2)^4} ,\begin{array}{c} t_2 \\0\end{array}\right)
\end{aligned}
\end{equation}

The correction term is such that
\begin{equation}
\left.\frac{1-t_1^2 t_2^4}{(1-t_1)^2 (1-t_2)^4}\right|_{t_1= 0} \!\!\!\!+~ \left.\frac{1-\left(t_1^{-1}t_2^4\right)^2 t_2^4}{(1-t_1^{-1}t_2^4)^2 (1-t_2)^4}\right|_{t_1= \infty}\!\!\!\! -~ \frac{1 - t_2^4}{(1-t_2)^4} ~=~  \frac{1 - t_2^8}{(1-t_2^4)(1-t_2)^4}~,
\end{equation}
gives the Hilber-Poincar\'e series $HS(\IP_{[4:1:1:1:1]}[8],t_2)$ associated with the singular threefold involved in the flop (see Ref.~\cite{Brodie:2021toe} for more details about flops for general complete intersections in products of projective spaces).

{\bfseries Complex structure dependence.} Since line bundle cohomology is not constant in families, it is natural to ask about the complex structure dependence of the proposed generating functions. Again, without making any general statements, we consider the case of Calabi-Yau hypersurfaces $X$ of bi-degree $(2,4)$ in $\IP^1\times \IP^3$ and study the way in which the generating functions change along a particular special locus in the complex structure moduli space. Let $f = x_0^2 f_0+x_0x_1 f_1+x_1^2 f_2$ be the defining polynomial for $X$, where $[x_0,x_1]$ are homogeneous coordinates on $\IP^1$ and $f_0,f_1,f_2$ are homogeneous polynomials of degree $4$ in the $\IP^3$ coordinates, denoted generically by $y$. For generic points $y\in\IP^3$, the equation $f(x,y)=0$ admits two points in $\IP^1$ as solutions. However, when $f_0(y)=f_1(y)=f_2(y)=0$, which for general $f_0$, $f_1$ and $f_2$ happens at 64 isolated points, the equation admits an entire $\IP^1$ as a solution. These isolated genus 0 curves correspond to the flopping curves of $X$. 

The situation is radically different when $f_1=0$ and $f_0,f_2$ remain general. While this choice still leads to smooth hypersurfaces, the change in the structure of the zeroth line bundle cohomology is drastic, as shown in \fref{fig:X7887tuned1_intro}. This is to be expected, as now the locus $f_0(y)=f_2(y)=0$ defines a curve of genus $33$ in $\IP^3$, which means that the above $64$ isolated collapsing curves have now coalesced into a surface, a $\IP^1$-fibration over a genus 33 curve, the flop being replaced by an elementary transformation~\cite{Katz:1996ht}. The collapsing divisor is rigid and its class is given by $\Gamma=-2H_1+4H_2$, where $H_1 = \mathcal O_{\IP^1\times \IP^3}(1,0)|_X$ and $H_2 = \mathcal O_{\IP^1\times \IP^3}(0,1)|_X$.

\begin{center}
\begin{figure}[H]
\begin{center}
\includegraphics[width=7.1cm]{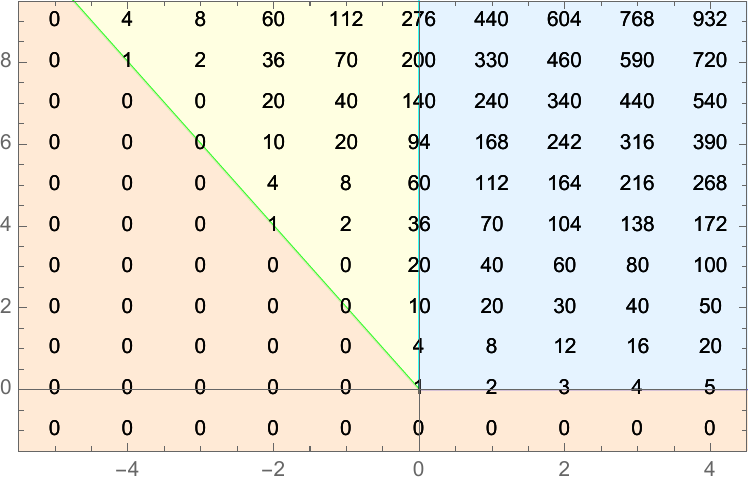}~~~~~
\includegraphics[width=7.2cm]{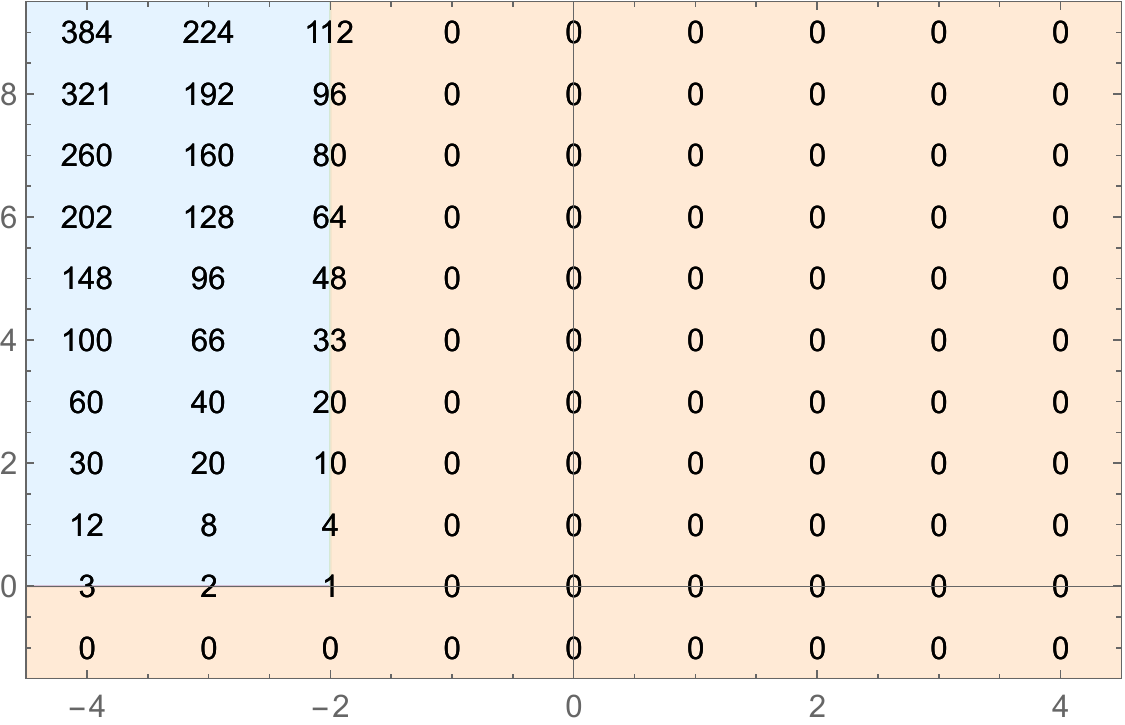}
\caption{\itshape Zeroth and first line bundle cohomology data (left and, respectively, right plot) for a Calabi-Yau hypersurface of bi-degree $(2,4)$ in $\mathbb{P}^1\times\mathbb{P}^3$ defined by the vanishing of $f = x_0^2 f_0+x_1^2 f_2$, with $f_0$ and $f_2$ general. 
}
\label{fig:X7887tuned1_intro}
\end{center}
\end{figure}
\end{center}

\begin{conjecture}\label{con:7887tuned_intro}
Let $X$ be a smooth hypersurface in $\IP^1\times \IP^3$ defined as the zero locus of a homogeneous polynomial $f = x_0^2 f_0+x_1^2 f_2$ where $[x_0,x_1]$ are homogeneous coordinates on $\IP^1$ and $f_0,f_2$ are general homogeneous polynomials of degree $4$ in the $\IP^3$ coordinates. 
A generating function for all line bundle cohomology dimensions relative to the basis $\{H_1,H_2\}$, where $H_1 = \mathcal O_{\IP^1\times \IP^3}(1,0)|_X$ and $H_2 = \mathcal O_{\IP^1\times \IP^3}(0,1)|_X$, is given by
\begin{equation}
\begin{aligned}
CS^0(X,\mathcal O_X) &= \left( \frac{(1-t_2^4)^2}{(1-t_1)^2(1-t_2)^4(1-t_1^{-2}t_2^4)} ,\begin{array}{cc} t_2& t_1 \\0&0 \end{array}\right)  \\
CS^1(X,\mathcal O_X) &= \left( \frac{(1-t_2^4)^2}{(1-t_1)^2(1-t_2)^4(1-t_1^{-2}t_2^4)} ,\begin{array}{cc} t_2& t_1 \\0&\infty \end{array}\right)  \\
-CS^2(X,\mathcal O_X) &= \left( \frac{(1-t_2^4)^2}{(1-t_1)^2(1-t_2)^4(1-t_1^{-2}t_2^4)} ,\begin{array}{cc} t_2& t_1 \\ \infty&0 \end{array}\right)  \\
-CS^3(X,\mathcal O_X) &= \left( \frac{(1-t_2^4)^2}{(1-t_1)^2(1-t_2)^4(1-t_1^{-2}t_2^4)} ,\begin{array}{cc} t_2& t_1 \\ \infty&\infty \end{array}\right) ~.
\end{aligned}
\end{equation}
\end{conjecture}

{\bfseries Complete intersection examples.}
The fact that the varieties covered by Conjectures~\ref{con:coh_intro1} and~\ref{con:7887tuned_intro} correspond to hypersurfaces is not essential for the existence of a universal generating function that can encode both the zeroth and the higher cohomology dimensions. The same is case in the following. 

\begin{conjecture}\label{con:7885_intro} Let $X$ be a general complete intersection of two hypersurfaces of bi-degrees $(1,1)$ and $(1,4)$ in $\IP^1\times \IP^4$, belonging to the deformation family with configuration matrix 
\begin{equation}\label{conf7885_intro}
\cicy{\IP^1 \\ \IP^4}{~1& 1~\\ ~1& 4~}~.
\end{equation}
The effective, movable and nef cones of $X$ are given by
\begin{equation}
\begin{aligned}
{\rm Eff}(X)=\IR_{\ge 0} H_1 +\IR_{\ge 0} &(H_2-H_1),~{\rm Mov}(X)= \IR_{\ge 0} H_1+\IR_{\ge 0} (4H_2-H_1)\\
&{\rm Nef}(X)=\IR_{\ge 0}H_1+\IR_{\ge 0}H_2~,
\end{aligned}
\end{equation}
where $H_1 = \mathcal O_{\IP^1\times \IP^4}(1,0)|_X$ and $H_2 = \mathcal O_{\IP^1\times \IP^4}(0,1)|_X$. All cohomology dimensions of all line bundles in the Picard group of $X$ are encoded by the following generating function, relative to the basis $\{H_1,H_2\}$ of ${\rm Pic}(X)$: 
\begin{equation}\label{cs07885_intro}
\begin{aligned}
CS^0(X,\mathcal O_X)&=\left( \frac{\left(1-t_2\right)^2\left(1-t_2^4\right)^2}{\left(1-t_1 \right)^2 \left(1-t_2 \right)^5 \left(1-t_1^{-1}t_2 \right) \left( 1-t_1^{-1}t_2^4 \right)} \,,\!\begin{array}{cc} t_2& t_1 \\0&0 \end{array}\!\! \right) \\
CS^1(X,\mathcal O_X)&=\left( \frac{\left(1-t_2\right)^2\left(1-t_2^4\right)^2}{\left(1-t_1 \right)^2 \left(1-t_2 \right)^5 \left(1-t_1^{-1}t_2 \right) \left( 1-t_1^{-1}t_2^4 \right)} \,,\!\begin{array}{cc} t_2& t_1 \\ \infty&0 \end{array}\!\! \right) \\
CS^2(X,\mathcal O_X)&=\left( \frac{\left(1-t_2\right)^2\left(1-t_2^4\right)^2}{\left(1-t_1 \right)^2 \left(1-t_2 \right)^5 \left(1-t_1^{-1}t_2 \right) \left( 1-t_1^{-1}t_2^4 \right)} \,,\!\begin{array}{cc} t_2& t_1 \\ 0&\infty \end{array}\!\! \right) \\
CS^3(X,\mathcal O_X)&=\left( \frac{\left(1-t_2\right)^2\left(1-t_2^4\right)^2}{\left(1-t_1 \right)^2 \left(1-t_2 \right)^5 \left(1-t_1^{-1}t_2 \right) \left( 1-t_1^{-1}t_2^4 \right)} \,,\!\begin{array}{cc} t_2& t_1 \\ \infty&\infty \end{array}\!\! \right)
\end{aligned}
\end{equation}
\end{conjecture}

\begin{center}
\begin{figure}[h]
\begin{center}
\includegraphics[width=7.1cm]{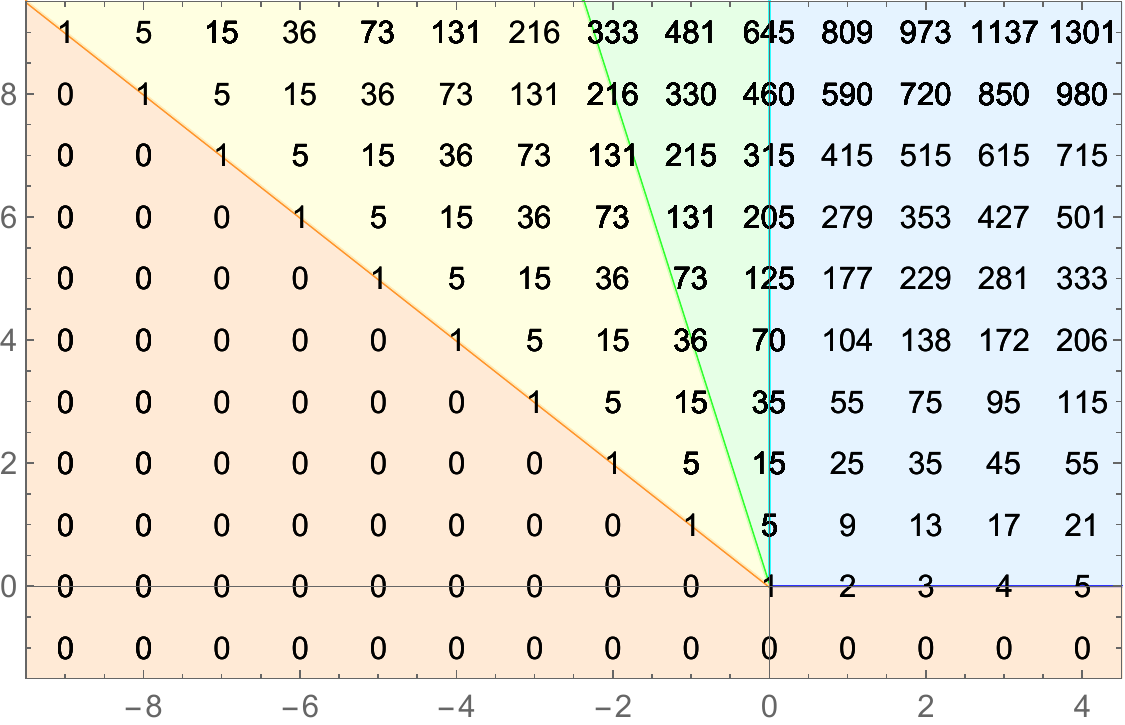}~~~~~
\includegraphics[width=7.1cm]{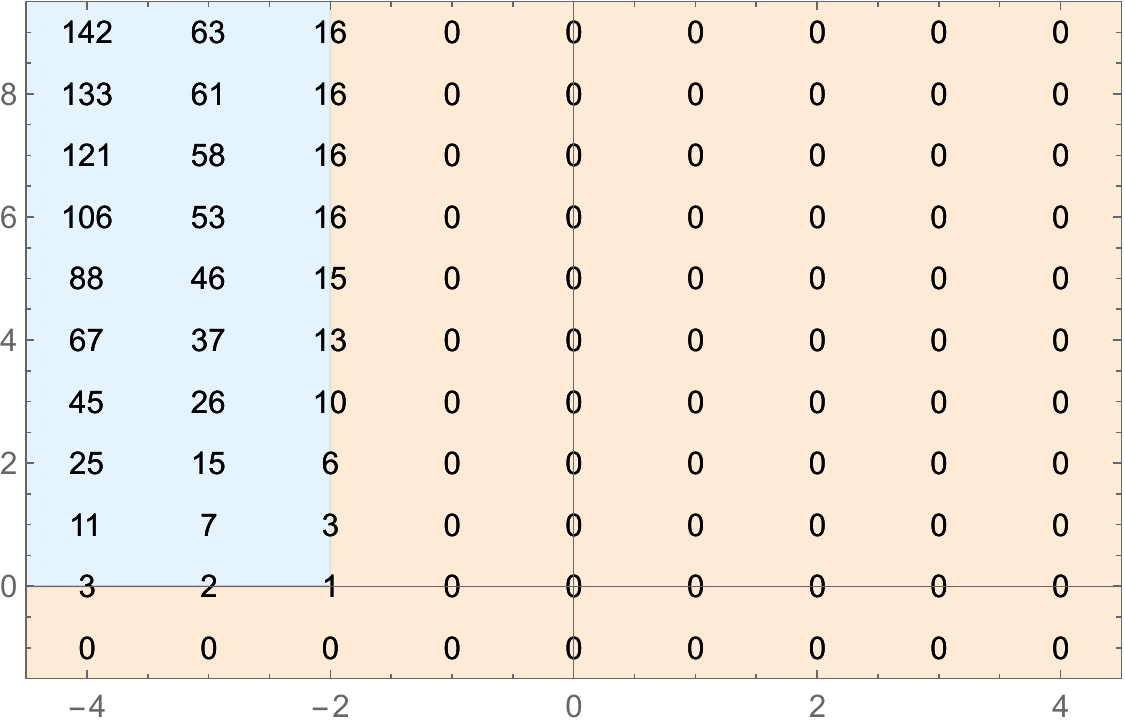}
\caption{\itshape Zeroth and first line bundle cohomology data (left plot and, respectively, right plot) for a general Calabi-Yau complete intersection in the deformation family~\eqref{conf7885_intro}. The numbers indicate cohomology dimensions while their locations indicate first Chern classes of line bundles.}
\label{fig:X7885_intro} 
\end{center}
\end{figure}
\end{center}

The generating function \eqref{cs07885_intro} has been constructed by adding up the Hilbert-Poincar\'e series associated with the coordinate rings of the two birational models of $X$ and subtracting a correction term, as discussed in Example~\ref{ex:7885}: 
\begin{equation*}
\begin{aligned}
CS^0(X,t_1,t_2) & =\left(\frac{(1-t_1t_2)(1-t_1t_2^4)}{(1-t_1)^2 (1-t_2)^5}~,\begin{array}{cc} t_2& t_1 \\0&0 \end{array}\right) + \\
&~~~~~~~~~~~~\left( \frac{(1-t_1^{-1}t_2^5)^2}{(1-t_1^{-1}t_2)^2 (1-t_2)^5(1- t_1^{-1}t_2^4)}~,\begin{array}{cc} t_2& t_1 \\0&0 \end{array}\right) -\left(\frac{1+t_2^5}{\left(1-t_2\right)^5}~,\begin{array}{c} t_2 \\0 \end{array}\right)~,
\end{aligned}
\end{equation*}
where the correction term is such that:
\begin{equation}
\begin{aligned}
\left.\frac{(1-t_1t_2)(1-t_1t_2^4)}{(1-t_1)^2 (1-t_2)^5}\right|_{t_1= 0}& \!\!+~ \left.\frac{(1-t_1^{-1}t_2^5)^2}{(1-t_1^{-1}t_2)^2 (1-t_2)^5(1- t_1^{-1}t_2^4)}\right|_{t_1= \infty } - \frac{1+t_2^5}{\left(1-t_2\right)^5} = HS(\IP^4[5],t_2)
\end{aligned}
\end{equation}
which is the Hilber-Poincar\'e series $HS(\IP^4[5],t_2)$ associated with the singular threefold involved in the flop. The generating function in~\eqref{cs07885_intro} should be interpreted as the Hilbert-Poincare\'e series associated with the Cox ring of $X$, represented as a complete intersection in a toric variety, as detailed in Example~\ref{ex:7885}. \fref{fig:X7885_intro} displays a part of the cohomology data on which Conjecture~\ref{con:7885_intro} is based.

{\bfseries The bicubic Calabi-Yau threefold.}
The above Calabi-Yau threefold examples suggest a certain pattern for constructing the cohomology series, namely to combine the Hilbert-Poincar\'e series associated with each of the birational models of the variety, while keeping track of the way in which the various Mori chambers attach to each other in the movable cone, and then subtract suitable correction terms, such that along every wall separating two Mori chambers the relevant contributions restrict to the Hilbert-Poincar\'e series of the singular variety involved in the corresponding small modification. While confirming this pattern for the zeroth cohomology series, the following example shows that this prescription is not general enough, if the aim is to obtain a universal generating function that encodes both the zeroth and the higher cohomology series.

\begin{conjecture}\label{con:bicubic_intro1}
Let $X$ be a general hypersurface of bidegree $(3,3)$ in $\IP^2\times \IP^{2}$. The effective, movable and nef cones coincide, 
\begin{equation}
{\rm Eff}(X)={\rm Mov}(X)={\rm Nef}(X)=\IR_{\ge 0}H_1+\IR_{\ge 0}H_2~,
\end{equation} 
where $H_1 = \mathcal O_{\IP^2\times \IP^2}(1,0)|_X$ and $H_2 = \mathcal O_{\IP^2\times \IP^2}(0,1)|_X$. The first line bundle cohomology is non-trivial within the two disconnected cones $\IR_{>0}H_2+\IR_{\geq 0}(-H_1+H_2)$ and $\IR_{>0} H_1+\IR_{\geq 0}(H_1-H_2)$. Defining
\begin{equation}
G(x,y) = 
\frac{(x^{-1} y)^3 ((1 + x - y)^3 - 1 + 3 x (1 - y))}{(1 - x^{-1} y)^3 (1 - y)^3}~,
\end{equation}
all cohomology dimensions of all line bundles on $X$ are encoded in the following generating functions, relative to the basis $\{H_1,H_2\}$ of ${\rm Pic}(X)$:
\begin{equation}
\begin{aligned}
CS^0(X,\mathcal O_X)&  = 1 +  \left( G(t_1,t_2)\,,\!\begin{array}{cc} t_1& t_2 \\0&0 \end{array}\!\!\right)+\left( G(t_2,t_1)\,,\!\begin{array}{cc} t_1& t_2 \\0&0 \end{array}\!\!\right) \\
CS^1(X,\mathcal O_X)& = 0 +  \left( G(t_1,t_2)\,,\!\begin{array}{cc} t_1& t_2 \\ \infty&0 \end{array}\!\!\right)+\left( G(t_2,t_1)\,,\!\begin{array}{cc} t_1& t_2 \\0&\infty \end{array}\!\!\right) \\
-CS^2(X,\mathcal O_X)& = 2+ \left( G(t_1,t_2)\,,\!\begin{array}{cc} t_1& t_2 \\ 0&\infty \end{array}\!\!\right)+\left( G(t_2,t_1)\,,\!\begin{array}{cc} t_1& t_2 \\ \infty& 0 \end{array}\!\!\right) \\
-CS^3(X,\mathcal O_X)& =  1+ \left( G(t_1,t_2)\,,\!\begin{array}{cc} t_1& t_2 \\ \infty&\infty \end{array}\!\!\right)+\left( G(t_2,t_1)\,,\!\begin{array}{cc} t_1& t_2 \\ \infty&\infty \end{array}\!\!\right) 
\end{aligned}
\end{equation}
\end{conjecture}

The interpretation of the generating function $G(x,y)$ is not transparent. However, the three terms that give the zeroth cohomology dimensions combine to
\begin{equation}
CS^0(X,\mathcal O_X) =\left( \frac{(1-t_1^3t_2^3)}{(1-t_1)^3(1-t_2)^3}\,,\!\begin{array}{cc} t_1& t_2 \\0&0 \end{array}\!\!\right)~,
\end{equation}
which, as expected, is the Hilbert-Poincar\'e series for the homogeneous coordinate ring of $X$. Indeed, in this case the global sections of all effective line bundles are polynomial and therefore are counted by the dimension of the corresponding graded piece of the homogeneous coordinate ring of $X$ (see the discussion around Example~\ref{ex:Cox}).

{\bfseries Non-Mori dram spaces with an infinite number of flops.}
The above examples are all Mori dream spaces, i.e.~varieties for which the Cox ring is finitely generated. Being a Mori dream space is a relatively strong condition and there are many examples of varieties for which this fails to be the case, e.g.~due to the non-polyhedrality or the non-closure of the effective or nef cones. Examples of non-Mori dream spaces that admit an infinite number of flops and line bundle cohomology thereon have been studied in Refs.~\cite{Brodie:2020fiq, Brodie:2021ain, Brodie:2021nit}. For such examples we propose generating functions involving an infinite sum of rational functions.

\begin{conjecture}\label{con:inf_flops_intro}
Let $X$ be a general complete intersection Calabi-Yau threefold in the deformation family given by the configuration matrix 
\begin{equation*}
\cicy{\IP^4 \\ \IP^4}{2& 0& 1 & 1& 1\\ 0& 2& 1& 1& 1}
\end{equation*}
and let $H_1 = \mathcal O_{\IP^4\times \IP^4}(1,0)|_X$ and $H_2 = \mathcal O_{\IP^4\times \IP^4}(0,1)|_X$.
The effective cone decomposes into a doubly infinite sequence of Mori chambers corresponding to the nef cones of isomorphic Calabi-Yau threefolds connected to $X$ through a sequence of flops, of the form
\begin{equation} 
K^{(n)} = \IR_{\geq 0} (a_{n+1} H_1-a_{n}H_2) + \IR_{\geq 0} (a_n H_1-a_{n-1}H_2)
\end{equation}
where $a_n$ is given by 
\begin{equation} 
a_n = \frac{\left(3+2 \sqrt{2}\right)^n-\left(3-2 \sqrt{2}\right)^n}{4 \sqrt{2}}~, ~~(a_n)=\ldots -204,-35,-6,-1,0,1,6,35,204,\ldots
\end{equation}
such that $K^{(0)}={\rm Nef}(X)$. A generating function for all line bundle cohomology dimensions can be written in the basis $\{H_1,H_2\}$ in terms of the functions
\begin{equation}
\begin{aligned}
G_n(t_1,t_2)= ~& \frac{ (1-(t_1^{a_{n+1}} t_2^{-a_{n}})^2) (1-(t_1^{a_{n}} t_2^{-a_{n-1}})^2)  (1-t_1^{a_n+a_{n+1}} t_2^{-a_{n-1}-a_{n}})^3}{ (1-t_1^{a_{n+1}} t_2^{-a_{n}})^5(1-t_1^{a_n} t_2^{-a_{n-1}})^5} \\[4pt]
C_n(t_1,t_2)=~ & \frac{(1-(t_1^{a_{n}} t_2^{-a_{n-1}})^2)(1-(t_1^{a_{n}} t_2^{-a_{n-1}})^3)}{(1-t_1^{a_n} t_2^{-a_{n-1}})^5}~,
\end{aligned}
\end{equation}
as follows:
\begin{equation}
\begin{aligned}
CS^0(X,\mathcal O_X)&  =  \left(\sum_{n=-\infty}^{0} G_n(t_1,t_2){+}C_n(t_1,t_2)\,,\!\begin{array}{cc} t_2& t_1 \\0&0 \end{array}\!\!\right)+\left(\sum_{n=1}^{\infty} G_n(t_1,t_2){+}C_n(t_1,t_2)\,,\!\begin{array}{cc} t_1& t_2 \\0&0 \end{array}\!\!\right) \\
CS^1(X,\mathcal O_X)&  =  \left(\sum_{n=-\infty}^{0} G_n(t_1,t_2){+}C_n(t_1,t_2)\,,\!\begin{array}{cc} t_2& t_1 \\0&\infty \end{array}\!\!\right) /.\,\{\text{remove terms } t_1^\alpha t_2^\beta \text{ with }\alpha+\beta<0 \} ~+ \\
&~~~~~~~~ \left(\sum_{n=0}^{\infty} G_n(t_1,t_2){+}C_n(t_1,t_2)\,,\!\begin{array}{cc} t_1& t_2 \\0&\infty \end{array}\!\!\right) /.\,\{\text{remove terms } t_1^\alpha t_2^\beta \text{ with }\alpha+\beta<0 \}  \\
1-CS^2(X,\mathcal O_X)&  = \left(\sum_{n=-\infty}^{0} G_n(t_1,t_2){+}C_n(t_1,t_2)\,,\!\begin{array}{cc} t_2& t_1 \\ \infty&0 \end{array}\!\!\right) /.\,\{\text{remove terms } t_1^\alpha t_2^\beta \text{ with }\alpha+\beta>0 \}~+\\
&~~~~~~~~ \left(\sum_{n=1}^{\infty} G_n(t_1,t_2){+}C_n(t_1,t_2)\,,\!\begin{array}{cc} t_1& t_2 \\ \infty&0 \end{array}\!\!\right)  /.\,\{\text{remove terms } t_1^\alpha t_2^\beta \text{ with }\alpha+\beta>0 \} \\
1-CS^3(X,\mathcal O_X)&  =  \left(\sum_{n=-\infty}^{0} G_n(t_1,t_2){+}C_n(t_1,t_2)\,,\!\begin{array}{cc} t_2& t_1 \\ \infty& \infty \end{array}\!\!\right)+\left(\sum_{n=0}^{\infty} G_n(t_1,t_2){+}C_n(t_1,t_2)\,,\!\begin{array}{cc} t_1& t_2 \\ \infty&\infty \end{array}\!\!\right)~. \\
\end{aligned}
\end{equation}

\end{conjecture}

While the examples discussed in this paper are  intriguing, the general picture remains mysterious. The rest of the paper discusses in more detail the above results and presents other similar examples. Importantly, neither the type of the variety, nor its dimension, Picard number or the assumption of generality for the defining equations seem to play an essential role for the existence of generating functions for line bundle cohomology dimensions. 

It is tempting to speculate about the potential relevance of generating functions for the classification of projective varieties. Since such functions carry a lot of numerical information about the variety, it is natural to ask whether they uniquely determine the variety. A similar conjecture has been proposed for Fano varieties stating that these are uniquely determined by their regularised quantum period, which is a generating function for certain Gromov-Witten invariants~\cite{Cotes:2021}. Gromov-Witten invariants as well as topological invariants of line bundles have already been used to classify the Calabi-Yau threefolds resulting from the Kreuzer-Skarke list of reflexive four-dimensional polytopes up to Picard number $7$ according to their diffeomorphism type~\cite{Chandra:2023afu, Gendler:2023ujl} and it is reasonable to expect that line bundle cohomology generating functions will further distinguish these and other varieties.

\newpage
\section{Generating functions for the Euler characteristic of holomorphic line bundles}\label{sec:EulerChar}
After reviewing some standard material about the Hilbert-Poincar\'e series, in this section we propose a generating function for the Euler characteristic of holomorphic line bundles over complex projective varieties realised as complete intersections in products of projective spaces or toric varieties.

\subsection{Hilbert functions and Hilbert polynomials}
The homogeneous coordinate ring of a projective variety $X\subset \IP^n$ is the quotient ring 
\begin{equation}
R = \mathbb C [x_0,x_1,\ldots x_n]/I~,
\end{equation}
where $I$ is the homogeneous ideal defining $X$ and $S=\mathbb C [x_0,x_1,\ldots x_n]$ is the polynomial ring in the corresponding variables. These rings and ideals are graded by the total degree of the polynomials. Concretely, as a vector space over $\IC$, the polynomial ring $S$ is infinite dimensional, however, it admits a direct sum decomposition into homogeneous pieces
\begin{equation}
S = \bigoplus_{m\geq 0} S_m~,
\end{equation}
where $S_m$ is the set of homogenous polynomials of total degree $m$ together with the zero polynomial, which forms a vector space of dimension $\binom{m+n}{m}$. Denoting $I_m = I\cap S_m$, the Hilbert function of $I$ is defined for non-negative $m$ by the expression
\begin{equation}
HF_{S/I}(m) = {\rm dim}_{\mathbb C}\, S_m/I_m.
\end{equation}
For instance, when $I$ is a monomial ideal, $HF_{S/I}(m)$ is the number of monomials of total degree $m$ that do not belong to $I$. The Hilbert function is eventually polynomial, i.e.~for large enough $m$ it agrees with some polynomial, known as the Hilbert polynomial. 

The difference between the Hilbert function and the Hilbert polynomial can be phrased in terms of line bundles. Let $L$ be a very ample line bundle on $X$ defining the embedding $i:X\rightarrow \IP^n$ where $n= h^0(X,L)-1$ and $L=i^\ast \mathcal O_{\IP^n}(1)$.
Then 
\begin{equation}
R = \bigoplus _{m\geq 0} H^{0}(X,L^{\otimes m}) = \mathbb C[x_0,x_1,\ldots x_n]/I~,
\end{equation}
where $I$ is the ideal of the image $i(X)$. The Hilbert function is 
\begin{equation}
HF_{R}(m) = h^0(X,L^{\otimes m})~,
\end{equation}
which corresponds to the dimension of the quotient of ${\rm Sym}^m(\mathbb C^{n+1})$ by the subspace of homogenous polynomials of degree $m$ which vanish on $i(X)$.
By Serre vanishing and the Riemann-Roch theorem, it is easy to see that for large enough $m$ the Hilbert function is eventually polynomial. 
Moreover, the Hilbert polynomial is precisely the Euler characteristic, $\chi(X, L^{\otimes m})$, which can be computed in terms of the first Chern class of the line bundle and the Todd classes of the tangent bundle of $X$ using the Hirzebruch-Riemann-Roch theorem, 
\begin{equation}
\chi(X,L) = \int_X {\rm ch}(L) \wedge {\rm td}(TX)~.
\end{equation}

The notion of Hilbert function can be easily generalised to the multi-graded setting. Given a multi-graded polynomial ring $S = \mathbb C[x_{1,0},\ldots x_{1,n_1}, x_{2,0},\ldots, \ldots,x_{p,n_p}]$ and a homogeneous ideal $I$, the multi-graded Hilbert function associated with $S/I$ is 
\begin{equation}
HF_{S/I}(m_1, m_2,\ldots,m_p) = {\rm dim}_{\mathbb C}\, S_{m_1,m_2,\ldots,m_p}/I_{m_1,m_2,\ldots,m_p}.
\end{equation}

\subsection{Hilbert-Poincar\'e series}\label{sec:HPseries}
It is convenient to encode the values of the Hilbert function $HF_R (m)$ for all $m$ simultaneously, into a generating function known as the Hilbert-Poincar\'e series of $R$, sometimes simply referred to as the Hilbert series. Using the same notation as above, the Hilbert-Poincar\'e series is defined as:
\begin{equation}\label{eq:HPseries}
HS(X,t) = \sum_{m=0}^{\infty} HF_R(m)\,t^m=\sum_{m=0}^{\infty} h^0(X,L^{\otimes m})\, t^m~,
\end{equation}
where the variable $t$ is called the ``counting variable''.
Note that for $m<0$, $h^0(X,L^{\otimes m})$ is always $0$, since $L$ is very ample, hence $HS(X,t)$ encodes the zeroth cohomology dimensions of all powers of $L$, both positive and negative. In general, the Hilbert-Poincar\'e series can be written in closed form as a rational generating function
\begin{equation}
HS(X,t) = \frac{p(t)}{(1-t)^{n+1}}~,
\end{equation}
where $n$ is the dimension of $X$ and $p(t)$ a polynomial in the counting variable~$t$ with integer coefficients.

\begin{proposition}
Let $X=\IP^n$ be the complex projective space of dimension $n$. The line bundle $L=\mathcal O_{\IP^n}(1)$ is very ample and leads to the Hilbert-Poincar\'e series
\begin{equation}\label{eq:HSPn}
HS(\IP^n,t) = \frac{1}{(1-t)^{n+1}} = \sum_{m=0}^{\infty} \binom{n+m}{n}t^m = \sum_{m=0}^{\infty}  h^0(\IP^n,\cO_{\IP^n}(m))\, t^m~.
\end{equation}
\begin{proof}
The statement follows immediately from Bott's formula for $h^0(\IP^n,\cO_{\IP^n}(m))$. Alternatively, it can be derived from the following elementary argument. Let $[x_0:x_1:\ldots:x_n]$ denote the homogeneous coordinates on~$\IP^n$. Then 
\begin{equation}
\frac{1}{1-x_i} = 1+x_i+x_i^2+\ldots~
\end{equation}
corresponds to the sum of every monomial in the variable $x_i$ counted exactly once. It follows that the product 
\begin{equation} \label{PnHS}
\prod_{i=0}^n\frac{1}{1-x_i} = \sum x_0^{m_0}x_1^{m_1}\ldots x_n^{m_n}
\end{equation}
contains all the monomials in $x_0,x_1,\ldots,x_n$ counted exactly once. Replacing $x_i\rightarrow t$ in Eq.~\eqref{PnHS}  for all $0\leq i\leq n$, the coefficient of $t^m$ on the RHS counts all the monomials of degree $m=\sum_i m_i$. 
\end{proof}
\end{proposition}

\begin{remark}
The Hilbert-Poincar\'e series~\eqref{eq:HSPn} encodes the zeroth cohomology dimensions of all line bundles on~$\IP^n$, since $\mathcal O_{\IP^n}(1)$ is very ample and hence $h^0(\IP^n,\cO_{\IP^n}(m))=0$ for $m<0$. The same series also encodes the top cohomology function, when expanded around $t=\infty$: 
\begin{equation}
\begin{aligned}
HS\left(\IP^n,\begin{array}{c} t\\ \infty \end{array}\right)& = \left(\frac{1}{(1-t)^{n+1}}~,\begin{array}{c} t\\ \infty \end{array}\right) =  \left(\frac{1}{(1-t^{-1})^{n+1}}~,\begin{array}{c} t\\ 0 \end{array}\right)/.~\{t\rightarrow t^{-1}\}\\
& = (-1)^{n+1}\sum_{m=0}^{\infty} \binom{n+m}{n}t^{-m-n-1} = (-1)^{n+1} \sum_{m=0}^{\infty}  h^n(\IP^n,\cO_{\IP^n}({-}m{-}n{-}1))\, t^{-m-n-1}~.
\end{aligned}
\end{equation}
More elementarily, the top cohomology representatives are in 1-to-1 correspondence with Laurent monomials of the form $x_0^{m_0}x_1^{m_1}\ldots x_n^{m_n}$, with $m_0,\ldots m_n<0$. Such monomials are counted by the series
\begin{equation}
\frac{1}{x_0\, x_1\ldots x_n}\left( 1+ \frac{1}{x_0}+ \frac{1}{x_0^2}+\ldots\right)\cdot\ldots \cdot \left( 1+ \frac{1}{x_n}+ \frac{1}{x_n^2}+\ldots\right)~.
\end{equation} 
By replacing $x_i=t^{-1}$, we recover the Hilbert-Poincar\'e series
\begin{equation}
t^{-(n+1)}HS(\IP^n,t^{-1}) = \sum_{m=0}^{\infty}  h^n(\IP^n,\cO_{\IP^n}({-}m{-}n{-}1))\, t^{-(m+n+1)}~.
\end{equation}
\end{remark}

\begin{example}
The Hilbert-Poincar\'e series associated with the weighted projective space of dimension $n$ with weights $a_0,\ldots,a_n$ is given by the formula:
\begin{equation}
HS(\mathbb{P}_{[a_0,\ldots,a_n]},t) = \frac{1}{\prod_{i=0}^{n}(1-t^{a_i})}= \sum_{m=0}^{\infty}  h^0(\mathbb{P}(a_0,\ldots,a_n),\cO_{\IP^n}(m))\, t^m~.
\end{equation}
\begin{proof}
Follow a similar argument with $[x_0:x_1:\ldots:x_n]$ denoting the homogeneous coordinates on $\IP_{[a_0,\ldots,a_n]}$. The coordinates $x_i$ are now to be replaced by $t^{a_i}$ in order to account for the scaling behaviour of the different coordinates.
\end{proof}
\end{example}

\begin{example}
Consider a smooth generic hypersurface $X\subset \IP^n$ of degree $d$. Let $n \geq 4$ such that, by the Lefschetz hyperplane theorem, Pic$(X)\cong \mathbb{Z}$ (see, for example, Section~3.1.B of Ref.~\cite{lazarsfeld2004positivity1}). Let $L=i^\ast\mathcal O_{\IP^n}(1)=\cO_X(1)$, where $i:X\hookrightarrow \IP^n$ is the embedding map, and note that $L$ is very ample by definition, so that $h^0(X,L^{\otimes m})=0$ for $m<0$. The Hilbert-Poincar\'e series of $X=\IP^n[d]$ with respect to $L$ encodes the dimensions of the spaces of global sections for all line bundles on $X$ and is given by
\begin{equation}
HS(\IP^n[d],t) = \frac{1-t^d}{(1-t)^{n+1}} = \sum_{m=0}^{\infty}  h^0(\IP^n[d],\cO_{\IP^n[d]}(m))\, t^m~.
\end{equation}
The expression follows from the additivity of the Hilbert-Poincar\'e series relative to the Koszul resolution
\begin{equation}
0\rightarrow \cO_{\IP^n}(-d)\rightarrow \cO_{\IP^n} \rightarrow \cO_X\rightarrow 0~,
\end{equation}
from which $t^d HS(\IP^n,t) - HS(\IP^n,t) + HS(\IP^n[d],t) =0$.
Note also that 
\begin{equation}
HS(\IP^n[d],t^{-1})=t^{n+1-d}HS(\IP^n[d],t)~,
\end{equation}
which corresponds to the Serre duality statement 
\begin{equation}
h^0(X,\cO_X(m))=h^{n-1} (X,\cO_X(-m)\otimes\cO_X(d-n-1))~.
\end{equation}
\end{example}

\subsection{Relation to the Euler characteristic of holomorphic line bundles}
The notion of Hilbert-Poincar\'e series can be easily extended to the multivariate case. The multivariate Hilbert-Poincar\'e series associated with the coordinate ring $S/I$ of a projective variety $X$ counts the number of multi-degree $m=(m_1,\ldots, m_p)$ monomials that are `linearly independent modulo' $I$, which only makes sense for $m_i\geq 0$. For sufficiently positive degrees, the Hilbert function becomes polynomial and corresponds to the Euler characteristic of certain line bundles. In this section we explore the possibility of extending the Hilbert-Poincar\'e series to a generating function that encodes the Euler characteristic of all possible line bundles on $X$. 

For concreteness, we consider projective varieties $X$ that can be embedded as smooth complete intersections in a product of $p$ projective spaces, $i:X\hookrightarrow \mathcal A = \IP^{n_1}\times \IP^{n_2}\times\ldots\times \IP^{n_p}$, 
cut out by $q$ polynomials with multidegrees $d^{(i)} = (d_{1}^{(i)}, d_{2}^{(i)},\ldots,d_{p}^{(i)})$, $1\leq i\leq q$. These multi-degrees can be conveniently collected as columns of the configuration matrix
\begin{equation}
\cicy{\IP^{n_1} \\ \vdots \\ \IP^{n_p}}{d_{1}^{(1)}& d_{1}^{(2)}& \ldots & d_{1}^{(q)}\\ \vdots & \vdots & &\vdots \\d_{p}^{(1)}& d_{p}^{(2)}& \ldots & d_{p}^{(q)}}~.
\end{equation}
For ease of notation, we write $d_{i,j} = d_i^{(j)}$. The dimension of $X$ is given by  $\sum_{i=1}^p n_i -q$ and in this paper we will assume this to be greater or equal to two, since for smooth projective curves of genus $g\geq 1$ the Picard group is uncountable.

The canonical bundle of the variety can be determined from the configuration matrix leading to following classification: $(i)$ when $\sum_{j=1}^q d_{i,j}<n_i+1$ for all $i$, $X$ is Fano; $(ii)$ when  $\sum_{j=1}^q d_{i,j}\leq n_i+1$ for all $i$, $X$ is semi-Fano; $(iii)$ when $\sum_{j=1}^q d_{i,j}=n_i+1$ for all $i$, $X$ is Calabi-Yau; $(iv)$ in all other cases, $X$ is of general type. 
We do not make any assumption on the type of $X$. If the embedding  $i:X\hookrightarrow \mathcal A$ is favourable, in the sense that $H^2(X,\mathbb Z)$ descends from $H^2(\mathcal A,\mathbb Z)$, the proposed generating function gives the Euler characteristic of all line bundles on $X$, otherwise it covers only those line bundles that descend from the ambient variety. Moreover, we make no assumption on the genericity of the defining polynomials, other than they should be general enough to define a sufficiently smooth variety (e.g.~mild singularities of $X$ may be allowed).   

Denoting by $H_i=\left( \cO_{\IP^{n_1}}(0)\otimes\ldots\otimes \cO_{\IP^{n_i}}(1) \otimes\ldots\otimes  \cO_{\IP^{n_p}}(0)\right)|_X$, the multivariate Hilbert-Poincar\'e series of $X$ takes the form
\begin{equation}\label{eq:HSCI}
\begin{aligned}
HS(X,t_1,t_2,\ldots, t_p) &= \frac{\prod_{j=1}^q \left(1- \prod_{i=1}^p t_i^{d_{i,j}} \right)  }{\prod_{i=1}^p \left(1- t_i \right)^{n_i+1}} ~,
\end{aligned}
\end{equation}
with the understanding that the expansion is taken around $t_i=0$. In general (though not for the above series since $d_{i,j}\geq 0$), the order in which a multivariate rational function is expanded as a power series matters. We adopt the convention by which the order of expansion is indicated by the order in which the variables appear as arguments. Moreover, we indicate the point of expansion below each variable, for instance Eq.~\eqref{eq:HSCI} will be written as:
\begin{equation}
HS\left(X,\begin{array}{cccc} t_1& t_2& \ldots& t_p\\0&0&\ldots& 0 \end{array}\right) = \frac{\prod_{j=1}^q \left(1- \prod_{i=1}^p t_i^{d_{i,j}} \right)  }{\prod_{i=1}^p \left(1- t_i \right)^{n_i+1}}~.
\end{equation}

In the previous sub-section we saw examples where the Hilbert-Poincar\'e series also encoded the information about the top cohomology, if expanded around infinity. Certainly, this is no surprise in virtue of Serre duality. For a complete intersection $X$, we can make this statement more concrete, as follows. Expanding around infinity involves replacing $t_i\rightarrow t_i^{-1}$, followed by the expansion around $\{t_i=0\}$ and then a second replacement $t_i\rightarrow t_i^{-1}$:
\begin{equation}
\begin{aligned}
HS&\left(X,\begin{array}{cccc} t_1& t_2& \ldots& t_p\\ \infty&\infty&\ldots& \infty \end{array}\right)  = \left( \frac{\prod_{j=1}^q \left(1- \prod_{i=1}^p t_i^{-d_{i,j}} \right)  }{\prod_{i=1}^p \left(1- t_i^{-1} \right)^{n_i+1}},\begin{array}{cccc} t_1& t_2& \ldots& t_p\\ 0&0&\ldots& 0 \end{array}\right)/.\{t_i\rightarrow t_i^{-1}\} \\
&= \left( \frac{(-1)^q \prod_{i=1}^p t_i^{-\sum_j d_{i,j}} }{(-1)^{\sum_i(n_i+1)} \prod_{i=1}^p t_i^{-(n_i+1)} } \frac{\prod_{j=1}^q \left(1- \prod_{i=1}^p t_i^{d_{i,j}} \right)  }{\prod_{i=1}^p \left(1- t_i \right)^{n_i+1}},\begin{array}{cccc} t_1& t_2& \ldots& t_p\\ 0&0&\ldots& 0 \end{array}\right)/.\{t_i\rightarrow t_i^{-1}\} \\[4pt]
& = \frac{\prod_{i=1}^p t_i^{K_i}}{(-1)^{{\rm dim}(X)+p}} \sum_{m_i\geq 0} h^0(X, H_1^{\otimes m_1}\otimes H_2^{\otimes m_2}\otimes\ldots \otimes H_p^{\otimes m_p})\, t_1^{-m_1}t_2^{-m_2}\ldots t_p^{-m_p}\\[4pt]
& = \frac{\prod_{i=1}^p t_i^{K_i}}{(-1)^{{\rm dim}(X)+p}} \sum_{m_i\leq 0} h^0(X, H_1^{\otimes (-m_1)}\otimes H_2^{\otimes (-m_2)}\otimes\ldots \otimes H_p^{\otimes (-m_p)})\, t_1^{m_1}t_2^{m_2}\ldots t_p^{m_p}\\[4pt]
& = \frac{1}{(-1)^{{\rm dim}(X)+p}} \sum_{K_i+\tilde m_i\leq 0} h^0(X, H_1^{\otimes (K_1-\tilde m_1)}\otimes H_2^{\otimes (K_2-\tilde m_2)}\otimes\ldots \otimes H_p^{\otimes (K_p-\tilde m_p)})\, t_1^{\tilde m_1}t_2^{\tilde m_2}\ldots t_p^{\tilde m_p}\\[4pt]
& = \frac{1}{(-1)^{{\rm dim}(X)+p}} \sum_{K_i+\tilde m_i\leq 0} h^n(X, H_1^{\otimes \tilde m_1}\otimes H_2^{\otimes \tilde m_2}\otimes\ldots \otimes H_p^{\otimes \tilde m_p})\, t_1^{\tilde m_1}t_2^{\tilde m_2}\ldots t_p^{\tilde m_p}~,
\end{aligned}
\end{equation}
where $K = \sum_{i=1}^p K_i H_i = \sum_{i=1}^p \left( \sum_{j=1}^q d_{ij}-(n_i+1)\right) H_i$ is the canonical bundle of $X$ and in the last step Serre duality was used. For large (negative) enough $\tilde m_i$, the coefficient of $t_1^{\tilde m_1}t_2^{\tilde m_2}\ldots t_p^{\tilde m_p}$ becomes a polynomial and is equal to the Euler characteristic $\chi(X, H_1^{\otimes \tilde m_1}\otimes H_2^{\otimes \tilde m_2}\otimes\ldots \otimes H_p^{\otimes \tilde m_p})$. It follows that the combined series 
\begin{equation}
HS\left(X,\begin{array}{cccc} t_1& t_2& \ldots& t_p\\0&0&\ldots& 0 \end{array}\right)  + HS\left(X,\begin{array}{cccc} t_1& t_2& \ldots& t_p\\ \infty&\infty&\ldots& \infty \end{array}\right) 
\end{equation}
gives the Euler characteristic for sufficiently large (both positive and negative) line bundle degrees. We take a step further and propose the following. 

\begin{con}\label{con:Euler} ${\rm (Conjecture~1)}$
Let $X\hookrightarrow \IP^{n_1}\times \IP^{n_2}\times\ldots\times \IP^{n_p}$ be a complete intersection in a product of $p$ projective spaces, cut out by $q$ homogeneous polynomials of multi-multidegrees $d^{(i)} = (d_{1,i}, d_{2,i},\ldots,d_{p,i})$ for $1\leq i\leq q$. Let $H_i=\left( \cO_{\IP^{n_1}}(0)\otimes\ldots\otimes \cO_{\IP^{n_i}}(1) \otimes\ldots\otimes  \cO_{\IP^{n_p}}(0)\right)|_X$ and
\begin{equation}
\begin{aligned}
HS(X,t_1,t_2,\ldots, t_p) &= \frac{\prod_{j=1}^q \left(1- \prod_{i=1}^p t_i^{d_{i,j}} \right)  }{\prod_{i=1}^p \left(1- t_i \right)^{n_i+1}} ~
\end{aligned}
\end{equation}
be the multivariate Hilbert-Poincar\'e series of $X$.
A generating function for the Euler characteristic of all line bundles on $X$ obtained as restrictions from the ambient variety is given by 
\begin{equation}
\sum_{\sigma\in S} (-1)^{{\rm sign}(\sigma)} HS\left(X,\begin{array}{ccc} t_1& \ldots& t_p\\\sigma(1)&\ldots& \sigma(p) \end{array}\right)  = \sum_{m_i\in \mathbb Z} \chi(X, H_1^{\otimes m_1}\otimes\ldots \otimes H_p^{\otimes m_p})\, t_1^{m_1}\ldots t_p^{m_p}
\end{equation}
where $S=\{0,\infty\}^{\times p}$ and ${\rm sign}(\sigma)$ is equal to the number of times $\infty$ appears in $\sigma$.
\end{con}

The conjecture may be extended to (complete intersections) in toric varieties. To discuss this, we use the homogeneous coordinate description of toric varieties~\cite{Cox1992, cox2011toric}. Let $V_{\Sigma}$ be an $n$-dimensional toric variety given by the fan $\Sigma$. Then $V_{\Sigma}= \left( \mathbb C^r\setminus Z_{\Sigma} \right)/ (\mathbb C^\ast)^{p}$~, where $p=r-n$ and $Z_{\Sigma}$ is the exceptional set fixed by a continuous subgroup of $(\mathbb C^\ast)^{p}$. The action of $(\mathbb C^\ast)^{p}$ is specified by a set of $p$ weighted scalings 
\begin{equation}
(z_1,z_2,\ldots,z_r)\rightarrow (\lambda^{q_{i,1}} z_1,\lambda^{q_{i,2}} z_2,\ldots, \lambda^{q_{i,r}} z_r)~,\qquad i = 1,\ldots, p~.
\end{equation}
The set of weights $q_{i,j}$ can be read off from the fan $\Sigma$ in the standard way. Conjecture~\ref{con:Euler} can be rephrased for complete intersections in toric varieties by replacing the Hilbert-Poincar\'e series for the ambient variety
\begin{equation}
HS( \IP^{n_1}\times\ldots\times \IP^{n_p}, t_1,\ldots,t_p) = \frac{1}{\prod_{i=1}^p \left(1- t_i \right)^{n_i+1}}
\end{equation}
with
\begin{equation}
HS(V_{\Sigma}, t_1,\ldots, t_{p}) =  \frac{1}{\prod_{i=1}^p \prod_{j=1}^r \left(1- t_i^{q_{i,j}} \right)}~,
\end{equation}
and interpreting the multi-degrees of the polynomials as weights under the $p$ scalings.

\section{Generating functions for line bundle cohomology dimensions}\label{sec:zeroth_coh}
\subsection{The Cox ring}
The Hilbert-Poincar\'e series associated with the homogeneous coordinate ring of a variety does not, in general, provide a generating function for the dimensions of the spaces of global sections of all line bundles. However, this can be achieved via the Cox ring, which plays the role of a total homogeneous coordinate ring. The concept was introduced by Hu and Keel in Ref.~\cite{Hu:2000} with the aim to generalise Cox' homogeneous coordinate ring construction for toric varieties to a broader class of varieties. 
The Cox ring of a projective variety $X$ is the graded ring consisting of all global sections of all line bundles on~$X$. More explicitly, we have the following. 

\begin{defn}
Let $X$ be a projective variety of dimension $n$ with $H^1(X,\mathcal O_X)=0$, such that $\text{Pic}^0(X)$ is trivial. Let $r$ denote the Picard number of $X$ and $L_1,L_2,\ldots L_r$ be line bundles on $X$ whose isomorphism classes form a $\mathbb Z$-basis of $\text{Pic}(X)$. The {\itshape Cox ring} is then defined as 
\begin{equation}
\text{Cox}(X) = \bigoplus_{(m_1,\ldots m_r)\in \mathbb Z^r} H^0(X,L_1^{\otimes m_1}\otimes\ldots\otimes L_r^{\otimes m_r})~,
\end{equation}
where the ring product is given by the multiplication of sections
\begin{equation}
H^0(X,L)\otimes H^0(X,L') \rightarrow H^0(X,L\otimes L')~.
\end{equation}
\end{defn}

\begin{remark}
The above definition depends on the choice of basis for ${\rm Pic}(X)$ and on the choice of particular line bundles in each isomorphism class. Nevertheless, all these choices give isomorphic rings. 
\end{remark}

\begin{example}
Let $X$ be a toric variety with free ${\rm Pic}(X)$. Then the Cox ring of $X$ coincides with its homogeneous coordinate ring. Conversely, the class of varieties whose Cox ring is polynomial precisely corresponds to toric varieties~\cite{Hu:2000}, see Ref.~\cite{Ottem} for more information. 
\end{example}

The Cox ring contains all the `coordinate rings' of $X$ as subrings in the following sense. If $L$ is a very ample divisor on $X$, defining the embedding $i:X\hookrightarrow \IP^n$, such that $ L =i^\ast \mathcal O_{\IP^n}(1)$, then 
\begin{equation}
\bigoplus_{m\in \mathbb Z}  H^0(X, L^{\otimes m}) \subset \text{Cox}(X)
\end{equation}
is a subring of $\text{Cox}(X)$ whose $\text{Proj}$ completely determines $X$. As such, the Cox ring contains information about the various projective embeddings that $X$ can admit.

The multivariate Hilbert-Poincar\'e series associated with~${\rm Cox}(X)$ is obtained by summing up the dimensions of its graded pieces and weighting them by the corresponding powers of the counting variables $t_1, t_2,\ldots t_r$:
\begin{equation}\label{eq:HSmultivar}
\begin{aligned}
HS(X,t_1,\ldots,t_r)&~ = \sum_{(m_1,\ldots, m_r)\in \mathbb Z^r} \text{dim}\,\text{Cox}(X)_{m_1,\ldots,m_r} t_1^{m_1}\ldots t_r^{m_r}\\
& ~= \sum_{(m_1,\ldots, m_r)\in \mathbb Z^r} h^0(X,L_1^{\otimes m_1}\otimes\ldots\otimes L_r^{\otimes m_r})\, t_1^{m_1}\ldots t_r^{m_r}~.
\end{aligned}
\end{equation}
In general, this is a formal Laurent series as negative exponents of the $t$-variables can appear. 
Single variable Hilbert-Poincar\'e series such as those discussed in Section~\ref{sec:HPseries} can be recovered from the expression~\eqref{eq:HSmultivar} by suitable monomial substitutions. For instance, if $L=L_1^{\otimes e_1}\otimes\ldots \otimes L_r^{e_r}$ in Eq.~\eqref{eq:HPseries}, the required substitution is $t=t_1^{e_1}\ldots t_r^{e_r}$.

\begin{example}\label{ex:Cox}
Let $X$ be a sufficiently general (e.g.~smooth or normal $\mathbb Q$-factorial) hypersurface in the product of projective spaces $\mathcal A=\IP^{n_1}\times\ldots \times \IP^{n_m}$ with $n_i\geq2$, defined by the vanishing of a homogeneous polynomial $f$. If~$D$ is a divisor on $\mathcal A$, the sequence
\begin{equation}
0\rightarrow H^0(\mathcal A, \cO(D-X))\rightarrow H^0(\mathcal A, \cO(D))\rightarrow H^0(X, \cO_X(D))\rightarrow 0
\end{equation}
is exact, since for $n_i\geq2$, $H^1(\mathcal A,L)=0$ for any line bundle $L$. It follows that the Cox ring of $X$ is the quotient 
\begin{equation}
{\rm Cox}(X) = \mathbb C[x^{(1)}_0, \ldots, x^{(1)}_{n_1},\ldots,\ldots, x^{(m)}_0,\ldots, x^{(m)}_{n_m}]/(f)~.
\end{equation}
\end{example}

\begin{defn}
$\mathbb Q$-factorial projective varieties whose Cox ring is finitely generated (for instance, any log-Fano variety is of this type~\cite{BCHM}) are called {\itshape Mori dream spaces}~\cite{Hu:2000}.  
\end{defn}

For Mori dream spaces, the cone of effective divisors is polyhedral and decomposes into finitely many convex sets called Mori chambers, each providing a birational model of the variety. More concretely, $X$ is a Mori dream space if it is a normal $\mathbb Q$-factorial projective variety, ${\rm Pic}(X)$ is finitely generated, ${\rm Nef}(X)$ is generated by the classes of finitely many semi-ample divisors, and there is a finite collection of small $\mathbb Q$-factorial modifications $\phi_i:X\dashrightarrow X_i $ such that for each $X_i$, ${\rm Nef}(X_i)$ is generated by the classes of finitely many semi-ample divisors and the movable cone of $X$ decomposes as $ {\rm Mov}(X)=\bigcup_i \phi_i^\ast {\rm Nef}(X_i)$.

The advantage of working with Mori dream spaces is that once a presentation for the Cox ring is known, this gives an embedding of $X$ into a simplicial toric variety $Y$ such that each of the modifications $\phi_i$ mentioned above is induced from a modification of the ambient toric variety $Y$, see~\cite{Hu:2000}. 
Obstructions to being a Mori dream space include non-polyherality of the effective/nef cones, as well as situations where the nef cone or the effective cone are not closed.

\subsection{Examples in dimension 2}

The birational structure of very general 
surfaces of bidegree $(d,e)$ in $\IP^1\times \IP^2$ has been described in Ref.~\cite{Ottem2015}. We focus on the following two cases: 
\begin{enumerate}
\itemsep0em
\item[$(i)$] $e=1$; in this case $X$ is a projective bundle over $\IP^1$, i.e.~a Hirzebruch surface;
\item[$(ii)$] $d=2$ and $e\geq3$; in this case, $X$ is a Mori dream space with ${\rm Pic}(X)\simeq{\rm Pic}(\IP^1\times\IP^2)\simeq\IZ^2$. The nef cone is equal to the effective cone and spanned by $H_1$ and $eH_2-H_1$, where $H_1 = \mathcal O_{\IP^1\times \IP^2}(1,0)|_X$ and $H_2 = \mathcal O_{\IP^1\times \IP^2}(0,1)|_X$.
\end{enumerate}
The other cases lead to surfaces with higher Picard rank: $(iii)$ $d=1$ and $e=2$ leads to a Picard number $5$ del Pezzo surface; $(iv)$  $d=2$ and $e=2$ leads to a del Pezzo surface with Picard number~$6$; or give non-Mori dream spaces: $(v)$ $d=1$ and $e\geq 3$ leads to a rational surface with infinitely many $(-1)$-curves; $(vi)$ $d\geq 3$ and $e\geq 2$ leads to a surface whose effective cone is not closed.

\subsubsection{Hirzebruch surfaces}\label{sec:Hirzebruch}

Let $\mathbb F_n$ be the $n$-th Hirzebruch surface $\IP(\cO_{\IP^1}\oplus \cO_{\IP^1}(n))$. 
Its Picard group can be written as ${\rm Pic}({\mathbb F_n}) = \mathbb Z C\oplus \mathbb Z F$, where $C$ is the unique irreducible curve with negative self intersection $C^2=-n$ and $F$ corresponds to the fiber with $F^2=0$, $F\cdot C = 1$. The nef cone is $${\rm Nef}({\mathbb F_n})=\IR_{\geq 0} F +\IR_{\geq 0} (nF+C)~,$$ while the effective cone is given by  $${\rm Eff}({\mathbb F_n})=\IR_{\geq 0} F +\IR_{\geq 0} C~.$$ 

The anticanonical class is $-K_{\mathbb F_n} = 2C+(n+2)F$. Using the Hirzebruch-Riemann-Roch formula, the Euler characteristic of a line bundle $\cO_{\mathbb F_n}(D = m_1 C+m_2F)$ is given by
\begin{equation}
\begin{aligned}
\chi({\mathbb F_n},\cO_{\mathbb F_n}(D)) &= \frac{1}{2}(D\cdot D - D\cdot K) + \chi({\mathbb F_n},\cO_{\mathbb F_n}) \\
& =\frac{1}{2}\left(-nm_1^2+2m_1m_2+(2-n)m_1+2m_2 \right)+1~.
\end{aligned}
\end{equation}

Following Ref.~\cite{Brodie:2019ozt}, the zeroth line bundle cohomology function on ${\mathbb F_n}$ can be written in closed form as 
\begin{equation}\label{coh0_Hirzebruch}
h^0({\mathbb F_n},\cO_{\mathbb F_n}(D))=\begin{cases}
\chi({\mathbb F_n},\cO_{\mathbb F_n}(D)) ~,\text{ if } D\in \text {Nef}({\mathbb F_n})\\[4pt]
\chi({\mathbb F_n},\displaystyle \cO_{\mathbb F_n}(D-\ceil{\frac{D\cdot C}{C\cdot C}}C))~,\text{ if } D\in\text{Eff}({\mathbb F_n})\setminus \text {Nef}({\mathbb F_n})~,\\
\end{cases}
\end{equation}
where $\ceil{D\cdot C/C\cdot C}= \ceil{m_1-m_2/n}$ when $D=m_1C+m_2F$ and the ceiling function is defined by $\lceil x\rceil =\min\{p\in \mathbb {Z} \mid p\geq x\}$. For illustration, Figure~\ref{dP1_h0} lists the values of the zeroth cohomology function on $\mathbb F_1$ for line bundles $m_1C+m_2F$ with $-3\leq m_1\leq 4$ and $-1\leq m_2\leq 7$.

\begin{figure}[htbp]
\begin{center}
     \includegraphics[width=7.1cm]{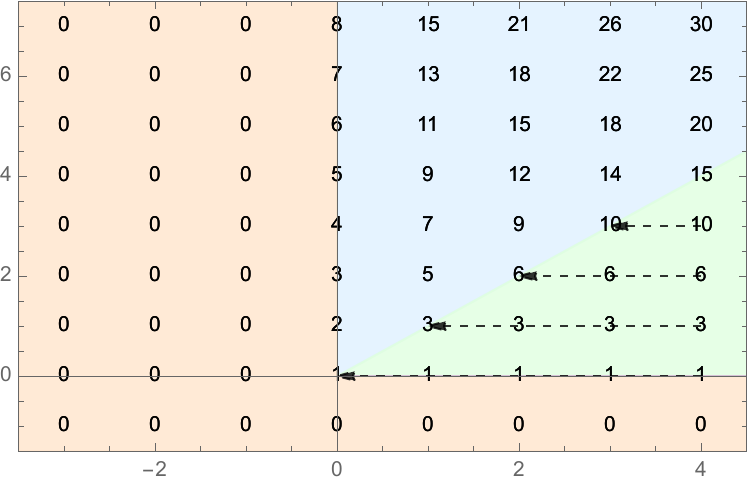}~~~~~
          \includegraphics[width=7.1cm]{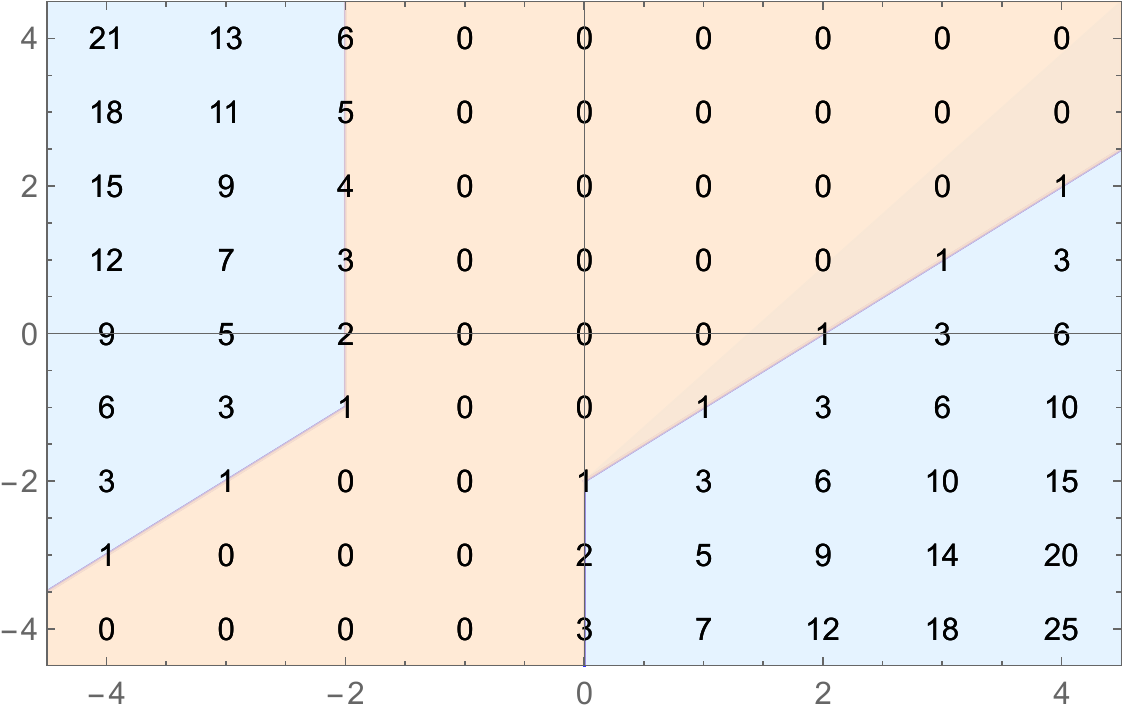}
\caption{\itshape Zeroth and first line bundle cohomology data (left plot and, respectively, right plot) for the Hirzebruch surface $\mathbb F_1$. The numbers indicate cohomology dimensions; the location in the plot indicates the first Chern class.}\label{dP1_h0}
\end{center}
\end{figure}

The top line bundle cohomology follows by Serre duality, 
\begin{equation}\label{coh2_Hirzebruch}
h^2({\mathbb F_n},\cO_{\mathbb F_n}(D))= h^0({\mathbb F_n},\cO_{\mathbb F_n}(-D+K))~,
\end{equation}
while the first cohomology follows from the Hirzebruch-Riemann-Roch formula
\begin{equation}\label{coh1_Hirzebruch}
h^1({\mathbb F_n},\cO_{\mathbb F_n}(D))= h^0({\mathbb F_n},\cO_{\mathbb F_n}(D))+h^0({\mathbb F_n},\cO_{\mathbb F_n}(-D+K))-\chi({\mathbb F_n},\cO_{\mathbb F_n}(D))~.
\end{equation}

\begin{proposition}
The middle and the zeroth line bundle cohomology functions on the $n$-th Hirzebruch surface are related by
\begin{equation}\label{coh0_coh1_Hirzebruch}
\begin{aligned}
h^1({\mathbb F_n},\cO_{\mathbb F_n}(m_1 C+m_2 F)) &= h^0({\mathbb F_n},\cO_{\mathbb F_n}(m_1 C+(n m_1-m_2-2) F))\\
 &+ h^0({\mathbb F_n},\cO_{\mathbb F_n}((-m_1-2) C+(-n m_1+m_2-n) F)) ~,
\end{aligned}
\end{equation}
where $C$ is the unique irreducible curve with negative self intersection $C^2=-n$ and $F^2=0$, $F\cdot C = 1$
\end{proposition}

\begin{proof}
The relation follows by direct computation using Eqns.~\eqref{coh0_Hirzebruch}, \eqref{coh2_Hirzebruch} and \eqref{coh1_Hirzebruch}.
\end{proof}
\begin{remark}
Formula~\eqref{coh0_coh1_Hirzebruch} is self-dual under Serre duality, as expected. Note also that at most one of the two terms can be non-zero for any given pair of integers $(m_1,m_2)$, which follows from the convexity of the effective cone and the nefness of the anticanonical class. 
\end{remark}

The Hilbert-Poincar\'e series associated with the Cox ring of $\mathbb F_n$ can be constructed from the toric data.  The fan contains four rays, as shown below. The divisor classes of $C$ and $F$ correspond to the rays $v_4$ and, respectively,~$v_1$.
\vspace{-12pt}
\begin{center}
\begin{tikzpicture}
\node at (1.2,0) {$v_1$};
\node at (0,1.2) {$v_2$};
\node at (-1.2,-2.2) {$v_3$};
\node at (0,-1.2) {$v_4$};
\draw[->] (0,0) -- (1,0);
\draw[->] (0,0) -- (0,1);
\draw[->] (0,0) -- (-1,-2);
\draw[->] (0,0) -- (0,-1);
\node at (7,0.8){$v_1 = (1,0),~v_2 = (0,1)$};
\node at (7,0){$v_3=(-1,-n),~v_4=(0,-1)$};
\node at (7,-.8){$v_2+v_4=0$};
\node at (7,-1.6){$v_1+nv_2+v_3=0$};
\end{tikzpicture}
\end{center}
The rays $v_i$ can be associated with coordinates $z_i$ with a weight system
\begin{equation}\label{eq:Fnweights}
\begin{array}{cccc}
z_1 & z_2& z_3& z_4\\
\hline
0& 1&0& 1\\
1& n&1&0
\end{array}
\end{equation}
From this, the Hilbert-Poincar\'e series associated with the Cox ring of $\mathbb F_n$ reads 
\begin{equation}\label{eq:HS0Fn}
\begin{aligned}
HS(\mathbb F_n, t_1,t_2) = \left(\frac{1}{(1-t_1 t_2^n)(1-t_2)^2(1-t_1)},\begin{array}{cc} t_1& t_2 \\0&0 \end{array}\right)~.
\end{aligned}
\end{equation}
As the homogeneous coordinate ring is the Cox ring, this gives the zeroth line bundle cohomology series relative to the basis $\{C,F\}$ of ${\rm Pic}(\mathbb F_n)$:
\begin{equation}\label{eq:HS0Fn}
\begin{aligned}
CS^0(\mathbb F_n, \mathcal O_{\mathbb F_n}))= HS(\mathbb F_n, t_1,t_2)~.
\end{aligned}
\end{equation}
We now turn to the higher cohomology dimensions. 
 
 \begin{thm} ${\rm(Theorem~2)}$ \label{thm:Hirz}
Let $\mathbb F_n$ be the $n$-th Hirzebruch surface 
with Picard group  ${\rm Pic}({\mathbb F_n}) = \mathbb Z C\oplus \mathbb Z F$, where $C$ is the unique irreducible curve with negative self intersection $C^2=-n$ and $F$ corresponds to the fiber with $F^2=0$, $F\cdot C = 1$. 
The cohomology dimensions of all line bundles in ${\rm Pic}({\mathbb F_n})$ are encoded by the following cohomology series relative to the basis $\{C,F\}$:
\begin{equation}
\begin{aligned}
CS^0(\mathbb F_n,\mathcal O_{\mathbb F_n})&{=} \left(\frac{1}{(1-t_1 t_2^n)(1-t_2)^2(1-t_1)},\begin{array}{cc} t_1& t_2 \\0&0 \end{array}\right)\\
CS^1(\mathbb F_n,\mathcal O_{\mathbb F_n})&{=} \left(\frac{1}{(1-t_1 t_2^n)(1-t_2)^2(1-t_1)},\begin{array}{cc} t_1& t_2 \\ \infty&0 \end{array}\right) +\left(\frac{1}{(1-t_1 t_2^n)(1-t_2)^2(1-t_1)},\begin{array}{cc} t_1& t_2 \\0& \infty\end{array}\right)\\
CS^2(\mathbb F_n,\mathcal O_{\mathbb F_n})&{=} \left(\frac{1}{(1-t_1 t_2^n)(1-t_2)^2(1-t_1)},\begin{array}{cc} t_1& t_2 \\\infty&\infty \end{array}\right)  
\end{aligned}
\end{equation}
\end{thm}

\begin{proof}
The Stanley-Reisner ideal of $\mathbb F_n$ is generated by $\langle z_1z_3,z_2z_4 \rangle$. Following the algorithm in Ref.~\cite{CohomOfLineBundles:Algorithm} (which forms the basis of the \texttt{cohomCalg} package~\cite{cohomCalg:Implementation}), zeroth cohomology representatives are in 1-to-1 correspondence with monomials of the form $z_1^{p_1}z_2^{p_2}z_3^{p_3}z_4^{p_4}$. Since the product $$\prod_{i=1}^4\frac{1}{1-z_i} = \sum z_1^{p_1}z_2^{p_2}z_3^{p_3}z_4^{p_4}$$ contains all the monomials in $z_1,z_2,z_3,z_4$ counted exactly once and given the toric weights~\eqref{eq:Fnweights}, by replacing $z_1\rightarrow t_2$, $z_3\rightarrow t_2$, $z_2\rightarrow t_1t_2^n$ and $z_4\rightarrow t_1$, Eq.~\eqref{eq:HS0Fn} follows.

The first cohomology representatives are in 1-to-1 correspondence with Laurent monomials of the form
\begin{equation}
\frac{z_1^{p_1}z_3^{p_3}}{z_2z_4 \cdot z_2^{p_2}z_4^{p_4}}~,~ \frac{z_2^{p_2}z_4^{p_4}}{z_1z_3 \cdot z_1^{p_1}z_3^{p_3}}
\end{equation}
with $p_i\geq 0$. The generating function that counts these monomials is 
\begin{equation}\label{GFnH1}
G(\mathbb F_n,t_1,t_2,t_1^{-1},t_2^{-1}) = t_1^{-2}t_2^{-n} \frac{1}{(1-t_2)^2}\frac{1}{(1-t_1^{-1}t_2^{-n})(1-t_1^{-1})} +t_2^{-2} \frac{1}{(1-t_2^{-1})^2}\frac{1}{(1-t_1t_2^{n})(1-t_1)} 
\end{equation}
where each fraction is to be expanded separately with either $t_1, t_2$ or, respectively, $t_1^{-1},t_2^{-1}$ small, followed by the multiplication of the corresponding series. This expression can be rearranged as follows. Concentrating on the first term, 
\begin{equation}
t_1^{-2}t_2^{-n} \frac{1}{(1-t_2)^2}\frac{1}{(1-t_1^{-1}t_2^{-n})(1-t_1^{-1})} = \frac{1}{(1-t_2)^2}\frac{1}{(1-t_1 t_2)(1-t_1)}~,
\end{equation}
where the first fraction is to be expanded around $t_2=0$ and the second fraction around $t_1=t_2=\infty$. However, the second fraction has the same expansion around $t_1=\infty$ and $t_2=0$, provided that the $t_1$ expansion is carried out first. With this, the first term in \eqref{GFnH1} corresponds to the Hilbert-Poincar\'e series 
\begin{equation}
 \frac{1}{(1-t_1 t_2^n)(1-t_2)^2(1-t_1)} 
\end{equation}
expanded around $t_1=\infty, t_2=0$. Similar arguments lead to the conclusion that the second term in~\eqref{GFnH1} corresponds to the Hilbert-Poincar\'e series expanded around $t_1=0, t_2=\infty$ (in this order). 

As such, the first line bundle cohomology series on $\mathbb F_n$ is given by the generating function 
\begin{equation}
CS^1(\mathbb F_n, \cO_{\mathbb F_n}) = HS\left(\mathbb F_n,\begin{array}{cc} t_1& t_2 \\ \infty&0 \end{array}\right)  + HS\left(\mathbb F_n,\begin{array}{cc} t_1& t_2 \\ 0& \infty \end{array}\right) ~.
\end{equation}

Finally, the top cohomology representatives are in 1-to-1 correspondence with Laurent monomials of the form 
\begin{equation}
\frac{1}{z_1z_2z_3z_4\cdot z_1^{p1}z_2^{p2}z_3^{p3}z_4^{p4}}~,
\end{equation}
which, following arguments similar to the above, are counted by the cohomology series
\begin{equation}
CS^2(\mathbb F_n,\cO_{\mathbb F_n}) = HS\left(\mathbb F_n,\begin{array}{cc} t_1& t_2 \\ \infty&\infty \end{array}\right)  ~.
\end{equation}
\end{proof}

\subsubsection{Surfaces of bidegree $(2,e\geq 3)$ in $\IP^1\times \IP^2$}\label{hypersurfacesP1P2}
For a very general hypersurface $X$ of bidegree $(2,e\geq 3)$ in $\IP^1\times \IP^2$, the Noether-Lefschetz theorem guarantees that ${\rm Pic}(X)\simeq{\rm Pic}(\IP^1 \times \IP^2)\simeq \IZ_2$. The case $e=3$ corresponds to a K3 surface with Picard number $2$, while the cases $e>3$ correspond to surfaces of general type.

The nef cone of $X$ has been shown in Ref.~\cite{Ottem2015} to be 
\begin{equation}
{\rm Nef}(X)={\rm Eff}(X)= \IR_{\geq 0} H_1 + \IR_{\geq0}(eH_2-H_1)~,
\end{equation}
where $H_1 = \mathcal O_{\IP^1\times \IP^2}(1,0)|_X$ and $H_2 = \mathcal O_{\IP^1\times \IP^2}(0,1)|_X$. Note that this is strictly larger than the nef cone of the ambient variety $\IP^1\times \IP^2$. However, $X$ admits an embedding as a complete intersection in a toric variety for which the nef cone descends from the ambient variety. 

To describe this embedding, suppose that $X\subset \IP^1\times \IP^2$ is cut out by the vanishing of the bihomogeneous polynomial
$$ f = x_0^2 f_0+x_0x_1 f_1+x_1^2 f_2~,$$
where $x_0, x_1$ are the coordinates on $\IP^1$ and $f_0, f_1, f_2$ are homogeneous polynomials of degree $e\geq 3$ in the coordinates $y_0, y_1, y_2$ of $\IP^n$. Consider the toric variety defined by the weight system
\begin{equation}
\begin{array}{ccccccc}
x_0&x_1&y_0&y_1&y_2&z_1 & z_2\\
\hline
1 &1 &0 &0 &0 &\!\!-1 &\!\!-1\\
0 &0 &1 &1 &1 &~e &~e
\end{array}
\end{equation}

Ottem has shown in Ref.~\cite{Ottem2015} that $X$ can be described as the complete intersection of three hypersurfaces in this toric variety, from which it inherits the nef and effective cones. This results in the following presentation of the Cox ring of $X$
\begin{equation}\label{coxringpres_surf}
{\rm Cox}(X)=\mathbb C[x_0,x_1,y_0,y_1,y_2,z_1,z_2]/I~,
\end{equation}
where $I=(f_0+x_1z_1,f_1-x_0z_1+x_1z_2,f_2-x_0z_2)$, which gives the generating function for the zeroth line bundle cohomology. 

\begin{proposition}
The multivariate Hilbert-Poincar\'e series associated with the Cox ring~\eqref{coxringpres_surf} corresponds to the rational function
\begin{equation}
HS(X,t_2,t_1)=\frac{(1-t_2^{e})^{3}}{(1-t_1)^2(1-t_2)^{3}(1-t_1^{-1}t_2^{e})^2}~.
\end{equation}
When expanded around $t_2=0$ and $t_1=0$ (in this order), the generating function gives the zeroth line bundle cohomology dimensions:
\begin{equation}
HS(X, t_2,t_1)=\sum_{m_1,m_2\in\mathbb Z} h^0(X,\cO_X(m_1 H_1+m_2 H_2)) t_1^{m_1}t_2^{m_2}~.
\end{equation}
\end{proposition}
\begin{proof}
The result follows directly from the presentation \eqref{coxringpres_surf} of the Cox ring.
\end{proof}

\begin{con}
The same rational function encodes the first and second cohomology dimensions:
\begin{equation}
\begin{aligned}
CS^1(X,\cO_X) & = HS \left( X\,,\!\begin{array}{cc} t_2& t_1 \\ \infty&0 \end{array}\!\!\right) + HS \left( X\,,\!\begin{array}{cc} t_2& t_1 \\ 0&\infty \end{array}\!\!\right)  =\sum_{m_1,m_2\in\mathbb Z} h^1(X,\cO_X(m_1 H_1+m_2 H_2)) t_1^{m_1}t_2^{m_2}~\\[4pt]
CS^2(X,\cO_X) & =HS \left( X\,,\!\begin{array}{cc} t_2& t_1 \\ \infty&\infty \end{array}\!\!\right)=\sum_{m_1,m_2\in\mathbb Z} h^{2}(X,\cO_X(m_1 H_1+m_2 H_2)) t_1^{m_1}t_2^{m_2}~,
\end{aligned}
\end{equation}
\end{con}

\subsection{Examples of Mori-dream spaces in dimension 3 and higher}
Finding an explicit presentation for the Cox ring of a projective variety is no easy task in general, as it essentially requires full information about the linear systems and divisors on $X$. In this section we focus on hypersurfaces of dimension $3$ and higher defined by the vanishing of a general bihomogeneous polynomial of bidegree $(d,e)$ in $\IP^m\times \IP^n$. The Picard group is given by the isomorphism ${\rm Pic}(X)\simeq {\rm Pic}(\IP^m\times \IP^n)\simeq\mathbb Z^2$.  

By Example~\ref{ex:Cox}, for $m,n\geq 2$, the Cox ring has a simple presentation as a quotient of the coordinate ring of the ambient variety by the defining polynomial. The effective cone and the nef cone coincide in this case and are equal to $\RR_{\ge 0}H_1+\RR_{\ge 0}H_2$. The generating function for the zeroth cohomology corresponds to the Hilbert-Poincar\'e series, that is 
\begin{equation}
CS^0(X,\cO_X)=HS(X,t_1,t_2) =  \left(  \frac{1-t_1^dt_2^e}{(1-t_1)^{m+1}(1-t_2)^{n+2}} \,,\!\begin{array}{cc} t_1& t_2 \\ 0&0 \end{array}\!\!\right)  ~.
\end{equation}
Writing down generating functions for the higher cohomologies is not straightforward and we will return to discussing an example in this class, namely the bicubic hypersurface in $\IP^2\times \IP^2$ later on. For now, we focus on the case $m=1$ which can be treated in a systematic way. 

\subsubsection{(Semi)-Fano, Calabi-Yau and general type hypersurfaces in $\IP^1\times \IP^n$}\label{hypersurfacesP1Pn}
Let $X$ be a hypersurface of bidegree $(d,e)$ in $\IP^1\times \IP^n$ defined by the vanishing of the bihomogeneous polynomial
$$ f = x_0^d f_0+x_0^{d-1}x_1 f_1+\ldots+x_1^d f_d~,$$
where $x_0, x_1$ are the coordinates on $\IP^1$ and $f_i$ are homogeneous polynomials in the coordinates $y_0,\ldots y_n$ of $\IP^n$. Let $H_1 = \mathcal O_{\IP^1\times \IP^n}(1,0)|_X$ and $H_2 = \mathcal O_{\IP^1\times \IP^n}(0,1)|_X$, such that ${\rm Pic}(X)=\mathbb Z H_1\oplus\mathbb Z H_2$.  All the cohomology series discussed in this section will be relative to the basis $\{H_1,H_2\}$. When $X$ is general, that is $X$ is smooth and the polynomials $f_i$ form a regular sequence, we have the following case distinction, as discussed in Ref~\cite{Ottem2015}:
\begin{enumerate}
\item[$(i)$] For $e=1$, $X$ is a $\IP^{n-1}$-bundle over $\IP^1$. 
\item[$(ii)$] For $d=1$, the second projection realizes $X$ as the blow-up of $\PP^n$ along $\{f_0=f_1=0\}$, and the exceptional divisor is linearly equivalent to $E=eH_2-H_1$. In this case 
$$
{\eff}(X)=\RR_{\ge 0} H_1+\RR_{\ge 0} E \mbox{ and }\mov(X)=\nef(X)=\RR_{\ge 0}H_1+\RR_{\ge 0}H_2.
$$
\item[$(iii)$] For $1<d< n$, there is a variety $X^+$ and a small birational modification $\phi:X\dashrightarrow X^+$, which induces a decomposition $\mov(X)=\nef(X)\cup \phi^*\nef(X^+)$. Also, $$
{\eff}(X)=\mov(X)=\RR_{\ge 0} H_1+\RR_{\ge 0} (eH_2-H_1)\mbox{ and }{\nef}(X)=\RR_{\ge 0}H_1+\RR_{\ge 0}H_2.
$$
\item[$(iv)$] For $d=n$, the divisor $eH_2-H_1$ is base-point free and defines a contraction to $\PP^{n-1}$. In this case 
$$\eff(X)=\mov(X)=\nef(X)=\RR_{\ge 0} H_1+\RR_{\ge 0} (eH_2-H_1).$$
\end{enumerate}
The remaining cases, i.e.~$d\ge n+1$ and $e\ge 2$, correspond to non-Mori dream spaces with $$\overline{\eff(X)}=\overline{\mov(X)}=\nef(X)=\RR_{\ge 0} H_1+\RR_{\ge 0} (neH_2-dH_1).$$
However, the divisor $neH_2-dH_1$ has no effective multiple.

In the cases when $X$ is a Mori dream space, i.e.~for $(d\leq n, e)$ and for $(d, e=1)$, Ottem has shown in Ref.~\cite{Ottem2015}  that the Cox ring admits the following presentation
\begin{equation}\label{coxringpres}
{\rm Cox}(X)=\mathbb C [x_0,x_1,y_0,\ldots,y_n,z_1,\ldots,z_d]/I~,
\end{equation}
where $I=(f_0+x_1z_1,f_1-x_0z_1+x_1z_2,\ldots,f_{d-1}-x_0z_{d-1}+x_1z_d,f_d-x_0z_d)$. 
Note that this class includes Fano varieties, for $d=1$ and $e\leq n$, almost Fano varieties for $d=2$ and $e\leq n$ or $d=1$ and $e=n+1$, Calabi-Yau varieties, for $d=2$ and $e=n+1$ and varieties of general type, in all other cases. 

The Hilbert-Poincar\'e series associated with the Cox ring gives a generating function for the zeroth line bundle cohomology when expanded around $t_2=0$ and $t_1=0$ (in this order):
\begin{equation}
CS^0(X,\cO_X)=  \left( \frac{(1-t_2^{e})^{d+1}}{(1-t_1)^2(1-t_2)^{n+1}(1-t_1^{-1}t_2^{e})^d} \,,\!\begin{array}{cc} t_2& t_1 \\ 0&0 \end{array}\!\!\right)  =\sum_{m_1,m_2\in\mathbb Z} h^0(X,\cO_X(m_1 H_1+m_2 H_2)) t_1^{m_1}t_2^{m_2}~.
\end{equation}

It is easy to see that all line bundle cohomologies except the zeroth, the first and the two top-most cohomologies vanish. This can be inferred from the exact sequence
\begin{equation}\label{Koszul}
0 \rightarrow \cO_{\IP^1\times \IP^n}(m_1-d,m_2-e)  \rightarrow \cO_{\IP^1\times \IP^n}(m_1,m_2) \rightarrow \cO_X(m_1H_1 + m_2 H_2)\rightarrow 0~
\end{equation}
and the fact that on $\IP^n$ only the zeroth and the $n$-th line bundle cohomologies can be non-trivial, while on $\IP^1$ only the zeroth and the first cohomologies can be non-trivial. 

\begin{con}\label{con:P1xPn}\label{conj:hypersurface} ${\rm (Conjecture~3)}$
Let $X$ be a general hypersurface of bi-degree $(d,e)$ in $\IP^1\times \IP^{n\geq 3}$ with $d\leq n$ and $e$ arbitrary or $d$ arbitrary and $e=1$. Denote $H_1 = \mathcal O_{\IP^1\times \IP^n}(1,0)|_X$ and $H_2 = \mathcal O_{\IP^1\times \IP^n}(0,1)|_X$. Then 
\begin{equation}\label{eq:hypersurfaces}
\begin{aligned}
CS^0(X,\mathcal O_X)&{=}\left( \frac{(1{-}t_2^{e})^{d+1}}{(1{-}t_1)^2(1{-}t_2)^{n+1}(1{-}t_1^{-1}t_2^{e})^d}\,,\!\begin{array}{cc} t_2& t_1 \\0&0 \end{array}\!\!\right) {=}\!\!\!\! \sum_{m_1,m_2\in \mathbb Z} \!\! h^0(X,\cO_X(m_1 H_1{+}m_2 H_2)) t_1^{m_1}t_2^{m_2}\\
CS^1(X,\mathcal O_X)&{=}\left( \frac{(1{-}t_2^{e})^{d+1}}{(1{-}t_1)^2(1{-}t_2)^{n+1}(1{-}t_1^{-1}t_2^{e})^d}\,,\!\begin{array}{cc} t_2& t_1 \\0&\infty \end{array}\!\!\right) {=}\!\!\!\! \sum_{m_1,m_2\in \mathbb Z} \!\! h^1(X,\cO_X(m_1 H_1{+}m_2 H_2)) t_1^{m_1}t_2^{m_2}\\
({-}1)^nCS^{n{-}1}(X{,}\mathcal O_X)&{=}\left( \frac{(1{-}t_2^{e})^{d{+}1}}{(1{-}t_1)^2(1{-}t_2)^{n+1}(1{-}t_1^{-1}t_2^{e})^d}\,,\!\begin{array}{cc} t_2& t_1 \\\infty&0 \end{array}\!\!\right) {=}\!\!\!\! \sum_{m_1,m_2\in \mathbb Z} \!\! h^{n-1}(X,\cO_X(m_1 H_1{+}m_2 H_2)) t_1^{m_1}t_2^{m_2}\\
(-1)^nCS^n(X,\mathcal O_X)&{=}\left( \frac{(1{-}t_2^{e})^{d+1}}{(1{-}t_1)^2(1{-}t_2)^{n+1}(1{-}t_1^{-1}t_2^{e})^d}\,,\!\begin{array}{cc} t_2& t_1 \\ \infty&\infty \end{array}\!\!\right) {=}\!\!\!\! \sum_{m_1,m_2\in \mathbb Z} \!\! h^n(X,\cO_X(m_1 H_1{+}m_2 H_2)) t_1^{m_1}t_2^{m_2}\\
\end{aligned}
\end{equation}
and all intermediate line bundle cohomologies vanish. 
\end{con}

\subsubsection{General hypersurfaces of bi-degree $(2,4)$ in $\IP^1\times \IP^3$}\label{sec:7887general}
While Conjecture~\ref{con:P1xPn} covers this case, it is interesting to look more in detail into the birational structure of this variety which 
has already been discussed in several places in the literature (see, e.g.,~Refs.~\cite{Brodie:2020fiq, Ottem2015}). Denoting $H_1 = \mathcal O_{\IP^1\times \IP^3}(1,0)|_X$ and $H_2 = \mathcal O_{\IP^1\times \IP^3}(0,1)|_X$, Lefschetz hyperplane theorem implies that ${\rm Pic}(X)=\mathbb Z H_1\oplus\mathbb Z H_2$. Moreover, 
\begin{equation}
{\eff}(X)=\mov(X)=\RR_{\ge 0} H_1+\RR_{\ge 0} (4H_2-H_1),\mbox{ and }{\nef}(X)=\RR_{\ge 0}H_1+\RR_{\ge 0}H_2~,
\end{equation}
hence the effective cone consists of two Mori chambers corresponding to the two birational models of~$X$ related by a flop~\cite{Brodie:2021toe}. For this variety, it turns out that both birational models belong to the same diffeomorphism class, such that their triple intersection numbers and second Chern class agree upon an integral change of basis. 
\begin{center}
\begin{figure}[h]
\begin{center}
\includegraphics[width=7.1cm]{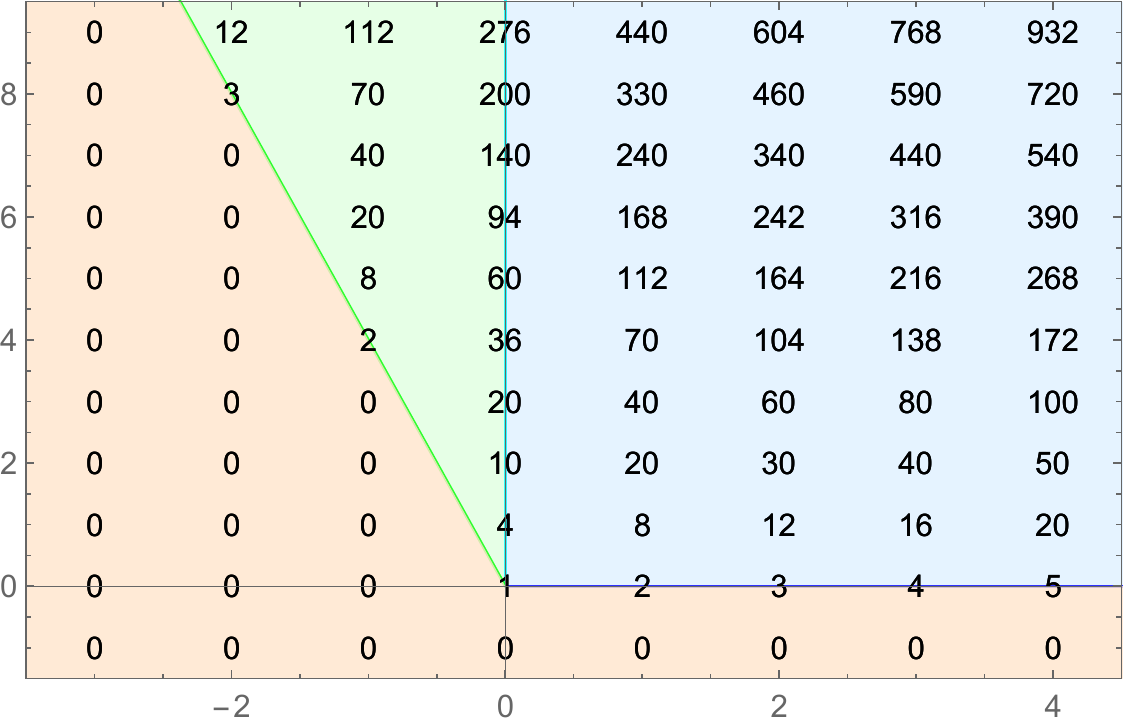}
\caption{\itshape Zeroth line bundle cohomology data for a generic Calabi-Yau hypersurface  of bi-degree $(2,4)$ in $\mathbb{P}^1\times\mathbb{P}^3$. The numbers indicate zeroth line bundle cohomology dimensions, while their locations in the plot indicate the first Chern class of the corresponding line bundles.}
\label{fig:X7887generic} 
\end{center}
\end{figure}
\end{center}

The structure of the effective cone as well as some data for the zeroth line bundle cohomology are shown in Figure~\ref{fig:X7887generic}, where the blue cone corresponds to the K\"ahler cone of $X$, denoted by $\mathcal K(X)$, and the green cone corresponds to the K\"ahler cone of the flopped manifold~$\mathcal K(X')$. Inside $\mathcal K(X)$, $h^0(X,L)=\chi(X,L)$ by Kodaira vanishing since $X$ is Calabi-Yau, while inside $\mathcal K(X')$, $h^0(X,L)=h^0(X',L')=\chi(X',L')$ where $L'$ is the line bundle on $X'$ corresponding to $L$ on~$X$, since dimensions of linear systems are preserved under a flop. The wall separating the two K\"ahler cones, excluding the origin, is covered by the Kawamata-Viehweg vanishing theorem and here $h^0(X,L)=\chi(X,L)=\chi(X',L')$. The line bundles lying along the two boundaries of the effective cone require a separate discussion and the dimensions of their zeroth cohomology can be obtained by sequence chasing. 
Outside of these regions the zeroth cohomology is trivial. The cohomology formulae are summarised below:
\begin{equation}
\begin{tabular}{ l | c}
 {\rm region in ${\rm Eff}(X)$}		&~$ h^0(X,L=m_1H_1+m_2H_2)$ \\[0pt]
\hline
$\cK(X)$ &~ $\chi(X,L)=2 m_1 (1 + m_2^2) + \frac{1}{3} m_2 (11 + m_2^2)$ \\[0pt]
$\cK(X')$ &~ $\chi(X',L')= \frac{32}{3}m_1(1-m_1^2) + 2 m_1 (1 + m_2^2) + \frac{1}{3} m_2 (11 + m_2^2)$ \\[0pt]
$\IR_{>0}H_2$ &~ $\chi(X,L)=\chi(X',L')$ \\[0pt]
$\IR_{>0}H_1$ & $\chi(\IP^1,\cO_{\IP^1}(m_1))$\\[0pt]
$\IR_{>0}(4H_2-H_1)~$ & $\chi(\IP^1,\cO_{\IP^1}(-m_1))$\\[0pt]
$m_1=m_2=0$ & $1$
\end{tabular}
\label{eq:7887_formulae}
\end{equation}

The presentation of the Cox ring of $X$ given in Ref.~\cite{Ottem2015} corresponds to an embedding as a complete intersection in a toric variety. To describe this, let $ f = x_0^2 f_0+x_0x_1 f_1+x_1^2 f_2$ be the  defining polynomial, where $[x_0,x_1]$ denotes homogeneous coordinates on $\IP^1$ and $f_0, f_1, f_2$ are homogeneous polynomials of degree $4$ in the coordinates $[y_0, y_1, y_2, y_3]$ of $\IP^3$. Then $X$ can be embedded as a complete intersection in the toric variety defined by the weight system
\begin{equation}
\begin{array}{cccccccc}
x_0&x_1&y_0&y_1&y_2&y_3&z_1 & z_2\\
\hline
1 &1 &0 &0&0 &0 &\!\!-1 &\!\!-1\\
0 &0 &1 &1&1 &1 &~4 &~4
\end{array}
\end{equation}
cut out by the ideal $I=(f_0+x_1z_1,f_1-x_0z_1+x_1z_2,f_2-x_0z_2)$, which gives
\begin{equation}\label{cox_7887}
{\rm Cox}(X)=\mathbb C[x_0,x_1,y_0,y_1,y_2,z_1,z_2]/I~,
\end{equation}
where $I=(f_0+x_1z_1,f_1-x_0z_1+x_1z_2,f_2-x_0z_2)$. 
 Hence we have the following.

\begin{crl}
The multivariate Hilbert-Poincar\'e series associated with the Cox ring~\eqref{cox_7887} corresponds to the generating function for the zeroth line bundle cohomology dimensions on a general hypersurface~$X$ of bi-degree $(2,4)$ in $\IP^1\times \IP^3$:
\begin{equation}\label{eq:cs07887}
\begin{aligned}
CS^0(X,\mathcal O_X )&=\left( \frac{(1-t_2^{4})^{3}}{(1-t_1)^2(1-t_2)^{4}(1-t_1^{-1}t_2^{4})^2},\begin{array}{cc} t_2& t_1 \\0&0 \end{array}\right) = \!\!\!\!\sum_{m_1,m_2\in \mathbb Z} h^0(X,\cO_X(m_1 H_1+m_2 H_2)) t_1^{m_1}t_2^{m_2}~.
\end{aligned}
\end{equation}
\end{crl}

\begin{remark}
The fact that the order of variable expansion matters is not surprising. This is essentially due to the presence of the factor $(1-t_1^{-1}t_2^{4})^2$ in the denominator, which leads to a different pole structure depending on the order in which the expansion is carried out. Expanding first around $t_1=0$ would give
\begin{equation}
\frac{1}{1-t_1^{-1}t_2^4}= \frac{t_1}{t_2^4}\frac{1}{t_1 t_2^{-4}-1}=-\frac{t_1}{t_2^4}\left(1+\frac{t_1}{t_2^4}+\frac{t_1^2}{t_2^{8}}+\ldots \right)~,
\end{equation}
while expanding first around $t_2=0$ and then around $t_1=0$ gives
\begin{equation} 
\frac{1}{1-t_1^{-1}t_2^4}= 1+\frac{t_2^4}{t_1}+\frac{t_2^{8}}{t_1^2}+\ldots~,
\end{equation}
which is the correct expansion in this case.
 \end{remark}
 
\begin{remark}
The generating function \eqref{eq:cs07887} can be decomposed into two contributions directly constructed from the Mori chamber decomposition of the effective cone and a correction term. As mentioned earlier on, the effective cone of a general hypersurface of bi-degree $(2,4)$ in $\IP^1\times \IP^3$ consists of two chambers associated with the nef cones of two (diffeomorphic) birational models connected by a flop: 
\begin{equation}
{\eff}(X)=\left( \RR_{\ge 0}H_1+\RR_{\ge 0}H_2\right) \cup \left( \RR_{\ge 0}H_2+\RR_{\ge 0}(4H_2-H_1)\right)~.
\end{equation}
The generating function \eqref{eq:cs07887} can then be obtained by summing up the Hilbert-Poincar\'e series associated with the coordinate rings of the two birational models of $X$: 
\begin{equation}\label{eq:cs07887v2}
\begin{aligned}
CS^0\left(X,\mathcal O_X \right)&= HS(X, t_1, t_2) + HS(X, t_1^{-1}t_2^4, t_2) - {\rm corr.~term}\\
& = \left( \frac{1-t_1^2 t_2^4}{(1-t_1)^2 (1-t_2)^4} ,\begin{array}{cc} t_2& t_1 \\0&0 \end{array}\right) + \left( \frac{1-\left(t_1^{-1}t_2^4\right)^2 t_2^4}{(1-t_1^{-1}t_2^4)^2 (1-t_2)^4} ,\begin{array}{cc} t_2& t_1 \\0&0 \end{array}\right) - \left( \frac{1 - t_2^4}{(1-t_2)^4} ,\begin{array}{c} t_2 \\0\end{array}\right) 
\end{aligned}
\end{equation}

The correction term can be recognised as $HS(\IP^3[4],t_2)$, however, this is not the right interpretation and the connection to~$X$ is more subtle:
\begin{equation}
\left.\frac{1-t_1^2 t_2^4}{(1-t_1)^2 (1-t_2)^4}\right|_{t_1= 0} + \left.\frac{1-\left(t_1^{-1}t_2^4\right)^2 t_2^4}{(1-t_1^{-1}t_2^4)^2 (1-t_2)^4}\right|_{t_1=\infty} - \frac{1 - t_2^4}{(1-t_2)^4} =  \frac{1 - t_2^8}{(1-t_2^4)(1-t_2)^4}~,
\end{equation}
which corresponds to  $HS(\IP_{[4:1:1:1:1]}[8],t_2)$, the Hilber-Poincar\'e series of the singular threefold obtained by contracting the $64$ isolated curves involved in the flop (see Ref.~\cite{Brodie:2021toe} for a discussion of flops for general complete intersections in products of projective spaces).

The generating functions for the higher cohomology dimensions follow the pattern proposed in Conjecture~\ref{con:P1xPn} (see also \fref{fig:X7887generic_intro} for data on the first cohomology dimensions).

\end{remark}

\subsubsection{Special hypersurfaces of bi-degree $(2,4)$ in $\IP^1\times \IP^3$}\label{sec:special2,4}
There are many ways in which the defining polynomial $f = x_0^2 f_0+x_0x_1 f_1+x_1^2 f_2$ can be specialised. Here, we consider one possibility and discuss the associated zeroth line bundle cohomology together with its generating function. 
At generic points $y\in\IP^3$, the equation  $f(x,y)=0$  admits two points in $\IP^1$ as solutions. When $f_0(y)=f_1(y)=f_2(y)=0$, which for generic $f_0,f_1$ and $f_2$ happens at $64$ isolated points in $\IP^3$, the equation admits an entire $\IP^1$ as a solution. These isolated genus 0 curves collapse to points as one approached the edge of the K\"ahler cone. 

The situation is radically different when $f_1=0$ and $f_0,f_2$ remain general. While this choice still leads to smooth hypersurfaces, the change in the structure of the zeroth line bundle cohomology is drastic, as shown in \fref{fig:X7887tuned1}. This is to be expected, as now the locus $f_0(y)=f_2(y)=0$ defines a curve of genus $33$ in $\IP^3$, which means that the $64$ isolated collapsing curves have coalesced into a surface, a $\IP^1$-fibration over a genus-33 curve, the flop being replaced by an elementary transformation~\cite{Katz:1996ht}. The collapsing divisor is rigid and its class is given by $\Gamma=-2H_1+4H_2$. 

\begin{center}
\begin{figure}[h]
\begin{center}
\includegraphics[width=7.1cm]{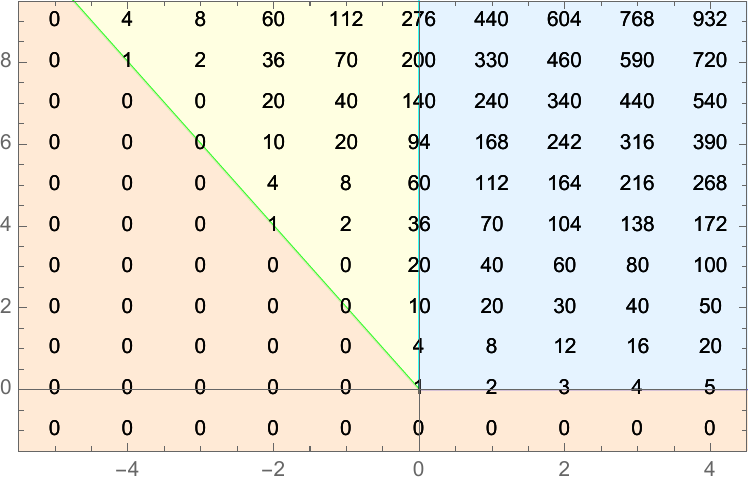}
\caption{\itshape Zeroth line bundle cohomology data for a Calabi-Yau hypersurface of bi-degree $(2,4)$ in $\mathbb{P}^1\times\mathbb{P}^3$ defined as the vanishing locus of $f = x_0^2 f_0+x_1^2 f_2$. The numbers indicate zeroth line bundle cohomology dimensions, while their locations in the plot indicate the first Chern class of the line bundles.}
\label{fig:X7887tuned1}
\end{center}
\end{figure}
\end{center}

As discussed in Ref.~\cite{Brodie:2021zqq}, the zeroth line bundle cohomology is captured by the following formulae, where $[C_1], [C_2]$ denote the curve classes dual to $H_1, H_2$, such that $[C_i]\cdot H_j = \delta_{ij}$: 
\begin{equation}
\begin{tabular}{ l | c}
 {\rm region in ${\rm Eff}(X)$}		&~$ h^0(X,L=\cO_X([D])=m_1H_1+m_2H_2)$ \\[4pt]
\hline
$\cK(X)\cup \IR_{>0}H_2$ &~ $\chi(X,L)$ \\[4pt]
$\IR_{>0}H_2+\IR_{>0}(4H_2-2H_1)$ & ~ $\chi\left(X, \cO_{X}\left([D]-\ceil{\frac{[D]\cdot [C_1]}{\Gamma\cdot [C_1]}}[\Gamma]\right)\right)$ \\[4pt]
$\IR_{>0}H_1$ & $\chi(\IP^1,\cO_{\IP^1}(m_1))$\\
$m_1=m_2=0$ & $1$
\end{tabular}
\label{eq:7887_formulae_tuned}
\end{equation}

These formulae are captured in the following generating function for the zeroth cohomology. 

\begin{con}
A generating function for the zeroth line bundle cohomology dimensions on a hypersurface $X$ in $\IP^1\times \IP^3$ defined as the zero locus of a homogeneous $f = x_0^2 f_0+x_1^2 f_2$ where $[x_0,x_1]$ are homogeneous coordinates on $\IP^1$ and $f_0,f_2$ are general homogeneous polynomials of degree $4$ in the $\IP^3$ coordinates is given by
\begin{equation}
\begin{aligned}
CS^0\left(X,\mathcal O_X \right) &= \left( \frac{1-t_1^2 t_2^4}{(1-t_1)^2 (1-t_2)^4} ,\begin{array}{cc} t_2& t_1 \\0&0 \end{array}\right) + \\
&~~~~~~~~\left( \frac{(1-t_1^{-2}t_2^8 )^2}{(1-t_1^{-2}t_2^4)(1-t_1^{-1}t_2^4)^2 (1-t_2)^4} ,\begin{array}{cc} t_2& t_1 \\0&0 \end{array}\right) - \left( \frac{1 - t_2^4}{(1-t_2)^4} ,\begin{array}{c} t_2 \\0\end{array}\right) ~\\
& = \left( \frac{(1-t_2^4)^2}{(1-t_1)^2(1-t_2)^4(1-t_1^{-2}t_2^4)} ,\begin{array}{cc} t_2& t_1 \\0&0 \end{array}\right)  
\end{aligned}
\end{equation}
\end{con}

Note that the first term and the correction term from Eq.~\eqref{eq:cs07887v2} remain the same, however, the second term, which is now associated with a Zariski chamber, is different. 
The higher line bundle cohomology dimensions are given by the following generating function (see also \fref{fig:X7887tuned1_intro} for data on the first cohomology dimensions).

\begin{con}
$({\rm Conjecture~4})$
Let $X$ be a smooth hypersurface in $\IP^1\times \IP^3$ defined as the zero locus of a homogeneous polynomial $f = x_0^2 f_0+x_1^2 f_2$ where $[x_0,x_1]$ are homogeneous coordinates on $\IP^1$ and $f_0,f_2$ are general homogeneous polynomials of degree $4$ in the $\IP^3$ coordinates. 
A generating function for all line bundle cohomology dimensions is given by
\begin{equation}
\begin{aligned}
CS^0(X,\mathcal O_X) &= \left( \frac{(1-t_2^4)^2}{(1-t_1)^2(1-t_2)^4(1-t_1^{-2}t_2^4)} ,\begin{array}{cc} t_2& t_1 \\0&0 \end{array}\right)  \\
CS^1(X,\mathcal O_X) &= \left( \frac{(1-t_2^4)^2}{(1-t_1)^2(1-t_2)^4(1-t_1^{-2}t_2^4)} ,\begin{array}{cc} t_2& t_1 \\0&\infty \end{array}\right)  \\
-CS^2(X,\mathcal O_X) &= \left( \frac{(1-t_2^4)^2}{(1-t_1)^2(1-t_2)^4(1-t_1^{-2}t_2^4)} ,\begin{array}{cc} t_2& t_1 \\ \infty&0 \end{array}\right)  \\
-CS^3(X,\mathcal O_X) &= \left( \frac{(1-t_2^4)^2}{(1-t_1)^2(1-t_2)^4(1-t_1^{-2}t_2^4)} ,\begin{array}{cc} t_2& t_1 \\ \infty&\infty \end{array}\right) ~.
\end{aligned}
\end{equation}
\end{con}

\subsubsection{Other examples of Mori-dream spaces in Picard number 2}

\begin{example}
Let $X$ be a general hypersurface of bi-degree $(3,5)$ in $\IP^1\times \IP^3$, which corresponds to a threefold of general type. This example is a particular case of Conjecture~\ref{conj:hypersurface}.
Denoting, as before, $H_1 = \mathcal O_{\IP^1\times \IP^3}(1,0)|_X$ and $H_2 = \mathcal O_{\IP^1\times \IP^3}(0,1)|_X$, it was shown in Ref.~\cite{Ottem2015} that the divisor $5H_2-H_1$ is base-point free and defines a contraction to $\IP^2$. Moreover,
\begin{equation}
{\rm Eff}(X)={\rm Mov}(X)={\rm Nef}(X)=\IR_{\ge 0} H_1+\IR_{\ge 0} (5H_2-H_1)~.
\end{equation}
Note that in this case ${\rm Nef}(X)$ does not descend from the ambient variety. 
A generating function for the zeroth line bundle cohomology throughout the entire Picard group can be written as: 
\begin{equation}
\begin{aligned}
CS^0\left(X,\mathcal O_X \right)&=\left( \frac{(1-t_2^{5})^{4}}{(1-t_1)^2(1-t_2)^{4}(1-t_1^{-1}t_2^{5})^3}~,\begin{array}{cc} t_2& t_1 \\0&0 \end{array}\right) = \!\!\!\!\sum_{m_1,m_2\in \mathbb Z} h^0(X,\cO_X(m_1 H_1+m_2 H_2)) t_1^{m_1}t_2^{m_2}\\
CS^1\left(X,\mathcal O_X \right)&=\left( \frac{(1-t_2^{5})^{4}}{(1-t_1)^2(1-t_2)^{4}(1-t_1^{-1}t_2^{5})^3}~,\begin{array}{cc} t_2& t_1 \\0&\infty \end{array}\right) = \!\!\!\!\sum_{m_1,m_2\in \mathbb Z} h^1(X,\cO_X(m_1 H_1+m_2 H_2)) t_1^{m_1}t_2^{m_2}\\
-CS^2\left(X,\mathcal O_X \right)&=\left( \frac{(1-t_2^{5})^{4}}{(1-t_1)^2(1-t_2)^{4}(1-t_1^{-1}t_2^{5})^3}~,\begin{array}{cc} t_2& t_1 \\\infty&0 \end{array}\right) = \!\!\!\!\sum_{m_1,m_2\in \mathbb Z} h^2(X,\cO_X(m_1 H_1+m_2 H_2)) t_1^{m_1}t_2^{m_2}\\
-CS^3\left(X,\mathcal O_X \right)&=\left( \frac{(1-t_2^{5})^{4}}{(1-t_1)^2(1-t_2)^{4}(1-t_1^{-1}t_2^{5})^3}~,\begin{array}{cc} t_2& t_1 \\\infty&\infty \end{array}\right) = \!\!\!\!\sum_{m_1,m_2\in \mathbb Z} h^3(X,\cO_X(m_1 H_1+m_2 H_2)) t_1^{m_1}t_2^{m_2}
\end{aligned}
\end{equation}
The generating function for the zeroth cohomology dimensions follows from the presentation for the Cox ring of~$X$ given in~\cite{Ottem2015}; the other generating functions are conjectural. 
\end{example}

\begin{example}\label{ex:7885}
Moving away from hypersurfaces, let$X$ be a general Calabi-Yau three-fold in the deformation family defined by the configuration matrix
\begin{equation}\label{conf7885a}
\cicy{\IP^1 \\ \IP^4}{\ 1~&1~ \\ \,1~&4~}~,
\end{equation}
with position $7885$ in the list of CICY threefolds and Hodge numbers $(h^{1,1}(X),h^{1,2}(X))=(2,86)$.
The structure of the effective cone has been studied in Ref.~\cite{Brodie:2020fiq}, namely
\begin{equation}
\begin{aligned}
{\rm Eff}(X)=\IR_{\ge 0} H_1 +\IR_{\ge 0} &(H_2-H_1),~{\rm Mov}(X)= \IR_{\ge 0} H_1+\IR_{\ge 0} (4H_2-H_1)\\
&{\nef}(X)=\RR_{\ge 0}H_1+\RR_{\ge 0}H_2~,
\end{aligned}
\end{equation}
such that the effective cone consists of three Mori chambers, one of which is a Zariski chamber and the other two are associated with the nef cones of the two (non-isomorphic) birational models of $X$. Figure~\ref{fig:7885plot} displays some data about the zeroth line bundle cohomology.
\begin{center}
\begin{figure}[h]
\begin{center}
\includegraphics[width=7.1cm]{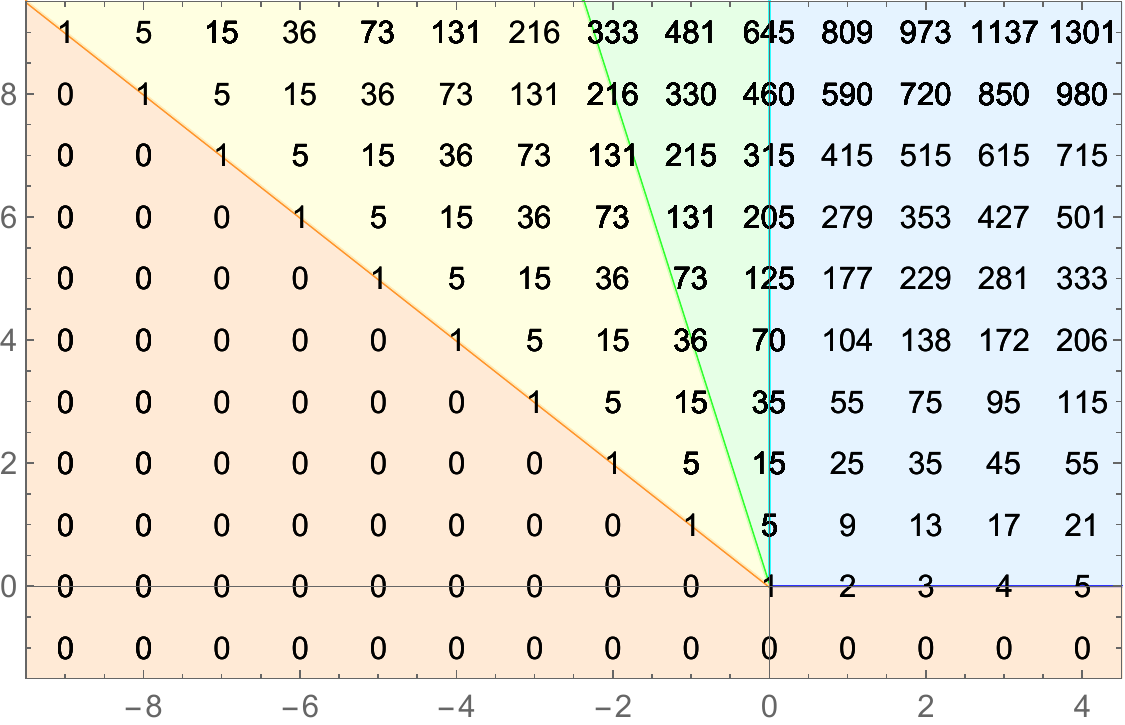}
\caption{\itshape Zeroth line bundle cohomology data for general complete intersections in the family \eqref{conf7885a}. The numbers indicate cohomology dimensions; the location in the plot indicates the first Chern class.}
\label{fig:7885plot}
\end{center}
\end{figure}
\end{center}

The Hilbert-Poincar\'e series associated with the coordinate ring of $X$ is 
\begin{equation}
HS(X,t_1,t_2) = \frac{\left(1-t_1t_2\right)\left(1-t_1t_2^4\right)}{\left(1-t_1\right)^2 \left(1-t_2\right)^5}~.
\end{equation}

To construct the zeroth cohomology series, note that $X$ can be flopped to a complete intersection $X'$ in a toric variety~\cite{Brodie:2021toe} with a weight system and weights for the defining equations given by
\begin{equation}
\fl{X} \sim 
\begin{array}{c c c c c| c c }
z_1  & z_2 	& y_1 & \ldots & y_5 ~ 	& \fl{P}_1  & \fl{P}_2\\
\hline
~1  & ~1 & 0	  & \ldots & 0						& 1  & 1  \\
\!\!\!-1 & \!\!\!-4 & 1 & \ldots & 1					& 0  & 0 
\end{array} ~~
 \sim ~~~~
\begin{array}{c c c c c| c c }
z_1  & z_2 	& y_1 & \ldots & y_5 ~ 	& \fl{P}_1  & \fl{P}_2\\
\hline
~ \!\!\!\!\!\!\!\!\!-1  & ~ \!\!\!\!\!\!\!\!\!-1 & 0	  & \ldots & 0						& -1  & -1  \\
4 & 1 & 1 & \ldots & 1					&~~ 5  &~~ 5 
\end{array} 
\end{equation}
which corresponds to the Hilbert-Poincar\'e series
\begin{equation}
HS(X',t_1,t_2) = \frac{(1-t_1^{-1}t_2^5)^2}{(1-t_1^{-1}t_2)^2 (1-t_2)^5(1- t_1^{-1}t_2^4)}~.
\end{equation}
Both $X$ and the flopped threefold $X'$ are resolutions of the same singular manifold $\contr{X}$ which belongs to the deformation family $\IP^4[5]$ as discussed in \cite{Brodie:2021toe}. As such, we construct the generating function for the zeroth line bundle cohomology on $X$ (and also on $X'$) from the following contributions
\begin{equation*}
\begin{aligned}
CS^0(X,t_1,t_2) & =\left(\frac{(1-t_1t_2)(1-t_1t_2^4)}{(1-t_1)^2 (1-t_2)^5}~,\begin{array}{cc} t_2& t_1 \\0&0 \end{array}\right) + \\
&~~~~~~~~~~~~\left( \frac{(1-t_1^{-1}t_2^5)^2}{(1-t_1^{-1}t_2)^2 (1-t_2)^5(1- t_1^{-1}t_2^4)}~,\begin{array}{cc} t_2& t_1 \\0&0 \end{array}\right) -\left(\frac{1+t_2^5}{\left(1-t_2\right)^5}~,\begin{array}{c} t_2 \\0 \end{array}\right)~,
\end{aligned}
\end{equation*}
where the correction term is such that:
\begin{equation}
\begin{aligned}
\left.\frac{(1-t_1t_2)(1-t_1t_2^4)}{(1-t_1)^2 (1-t_2)^5}\right|_{t_1= 0}& \!\!+~ \left.\frac{(1-t_1^{-1}t_2^5)^2}{(1-t_1^{-1}t_2)^2 (1-t_2)^5(1- t_1^{-1}t_2^4)}\right|_{t_1= \infty } - \frac{1+t_2^5}{\left(1-t_2\right)^5} = HS(\IP^4[5],t_2)
\end{aligned}
\end{equation}
Note that the information about the Zariski chamber and the Mori chamber associated with the nef cone of $X'$ is encoded in the second term, which reflects the fact that the dimensions of linear systems of divisors belonging to the Zariski chamber remain constant upon removing the fixed parts. 

The three contributions can be combined into a single rational function, such that:
\begin{equation}\label{gf7885}
\begin{aligned}
CS^0(X,t_1,t_2) & =\left(\frac{\left(1-t_2\right)^2\left(1-t_2^4\right)^2}{\left(1-t_1\right)^2 \left(1-t_2\right)^5\left(1-t_1^{-1}t_2\right)\left(1-t_1^{-1}t_2^4\right)},\begin{array}{cc} t_2& t_1 \\0&0 \end{array}\right) 
\end{aligned}
\end{equation}
Interpreted as the Hilbert-Poincar\'e series associated with the Cox ring of $X$, this rational function leads to a representation of $X$ as a complete intersection in a toric variety with a weight system and weights for the defining equations given by
\begin{equation}
\begin{array}{c c c c c c c| c c c c }
x_1  & x_2 	& y_1 & \ldots & y_5 ~ & z_1  & z_2 	& Q_1  & Q_2& Q_3  & Q_4\\
\hline
1  & 1 & 0	  & \ldots & 0		& \!\!\!\!-1&\!\!\!\!-1				& 0 & 0& 0 & 0  \\
0 & 0 & 1 & \ldots & 1		&1&4			& 1 & 1 & 4&4
\end{array} 
\end{equation}
In this representation, the flop $X\xdasharrow{~~~~} X'$ is induced from a flop of the ambient toric variety. The generating function \eqref{gf7885} also encodes the higher cohomologies (see \fref{fig:X7885_intro} for some data on the first cohomology).

\begin{con}\label{con:7885} ${\rm (Conjecture~5)}$ Let $X$ be a general complete intersection of two hypersurfaces of bi-degrees $(1,1)$ and $(1,4)$ in $\IP^1\times \IP^4$, belonging to the deformation family with configuration matrix 
\begin{equation}\label{conf7885}
\cicy{\IP^1 \\ \IP^4}{~1& 1~\\ ~1& 4~}~.
\end{equation}
Let $(H_1,H_2)$ be the basis of ${\rm Pic}(X)$ where $H_1 = \mathcal O_{\IP^1\times \IP^4}(1,0)|_X$ and $H_2 = \mathcal O_{\IP^1\times \IP^4}(0,1)|_X$. In this basis, the cohomology series are encoded by the following generating function: 
\begin{equation}\label{cs07885}
\begin{aligned}
CS^0(X,\mathcal O_X)&=\left( \frac{\left(1-t_2\right)^2\left(1-t_2^4\right)^2}{\left(1-t_1 \right)^2 \left(1-t_2 \right)^5 \left(1-t_1^{-1}t_2 \right) \left( 1-t_1^{-1}t_2^4 \right)} \,,\!\begin{array}{cc} t_2& t_1 \\0&0 \end{array}\!\! \right) \\
CS^1(X,\mathcal O_X)&=\left( \frac{\left(1-t_2\right)^2\left(1-t_2^4\right)^2}{\left(1-t_1 \right)^2 \left(1-t_2 \right)^5 \left(1-t_1^{-1}t_2 \right) \left( 1-t_1^{-1}t_2^4 \right)} \,,\!\begin{array}{cc} t_2& t_1 \\ \infty&0 \end{array}\!\! \right) \\
CS^2(X,\mathcal O_X)&=\left( \frac{\left(1-t_2\right)^2\left(1-t_2^4\right)^2}{\left(1-t_1 \right)^2 \left(1-t_2 \right)^5 \left(1-t_1^{-1}t_2 \right) \left( 1-t_1^{-1}t_2^4 \right)} \,,\!\begin{array}{cc} t_2& t_1 \\ 0&\infty \end{array}\!\! \right) \\
CS^3(X,\mathcal O_X)&=\left( \frac{\left(1-t_2\right)^2\left(1-t_2^4\right)^2}{\left(1-t_1 \right)^2 \left(1-t_2 \right)^5 \left(1-t_1^{-1}t_2 \right) \left( 1-t_1^{-1}t_2^4 \right)} \,,\!\begin{array}{cc} t_2& t_1 \\ \infty&\infty \end{array}\!\! \right)
\end{aligned}
\end{equation}
\end{con}

\begin{example} 
In this example $X$ is a general complete intersection Calabi-Yau (CICY) threefold belonging to the deformation family with configuration matrix 
\begin{equation}
\cicy{\IP^2 \\ \IP^5}{0& 0& 2 & 1\\ 2& 2& 1& 1}~.
\end{equation}
In the list of CICY threefolds~\cite{Candelas:1987kf}, this configuration matrix appears at position 7643, and corresponds to the pair of  Hodge numbers $(h^{1,1}(X),h^{1,2}(X))=(2,46)$.
As discussed in Ref.~\cite{Brodie:2021nit}, 
\begin{equation}
{\eff}(X)=\mov(X)=\RR_{\ge 0} H_1+\RR_{\ge 0} (3H_2-H_1),\mbox{ and }{\nef}(X)=\RR_{\ge 0}H_1+\RR_{\ge 0}H_2~,
\end{equation}
where $H_1 = \mathcal O_{\IP^2\times \IP^5}(1,0)|_X$ and $H_2 = \mathcal O_{\IP^2\times \IP^5}(0,1)|_X$. The structure of the effective cone is given by
\begin{equation}
{\eff}(X)=\left(\RR_{\ge 0}H_1+\RR_{\ge 0}H_2 \right) \cup \left(\RR_{\ge 0} H_2+\RR_{\ge 0} (3H_2-H_1)\right)~,
\end{equation}
the two Mori chambers corresponding to two isomorphic birational models of $X$. 
The zeroth line bundle cohomology series is given by the following generating function:
\begin{equation}
\begin{aligned}
CS^0(X,t_1,t_2) & = HS(X,t_1,t_2)+HS(X,t_1^{-1}t_2^3,t_2) - {\rm corr.~term}\\
&=\left(\frac{ (1-t_2^2)^2 (1-t_1^2 t_2) (1-t_1 t_2)}{(1-t_1)^3  (1-t_2)^6}~,\begin{array}{cc} t_2& t_1 \\0&0 \end{array}\right) + \\
&~~~~~~~~~~~~\left(\frac{(1-t_2^2)^2 (1-(t_1^{-1} t_2^3)^2 t_2) (1-(t_1^{-1}t_2^3)t_2)}{(1-t_1^{-1}t_2^3)^3 (1-t_2)^6 }~,\begin{array}{cc} t_2& t_1 \\0&0 \end{array}\right) -\left(\frac{(1-t_2^2)^3}{(1-t_2)^6}~,\begin{array}{c} t_2 \\0 \end{array}\right)~.
\end{aligned}
\end{equation}

The correction term is such that:
\begin{equation}
\begin{aligned}
\left.\frac{ \left(1-t_2^2\right)^2 \left(1-t_1^2 t_2\right) \left(1-t_1 t_2\right)}{\left(1-t_1\right)^3  \left(1-t_2\right)^6}\right|_{t_1= 0}& \!\!+~ \left.\frac{(1-t_2^2)^2 (1-(t_1^{-1} t_2^3)^2 t_2) (1-(t_1^{-1}t_2^3)t_2)}{(1-t_1^{-1}t_2^3)^3 (1-t_2)^6 }\right|_{t_1= \infty } - \frac{(1-t_2^2)^3}{(1-t_2)^6} \\
&  =  \frac{(1 - t_2^4)(1-t_2^2)^2}{(1-t_2^2)(1-t_2)^6}= HS(\IP_{[2:1:1:1:1:1:1]}[4~ 2~ 2],t_2)~,
\end{aligned}
\end{equation}
which is the Hilbert-Poincar\'e series 
associated with the singular variety involved in the flop~\cite{Brodie:2021toe}. 
\end{example}

The three contributions can be combined into a single generating function $G(X,t_1,t_2)$, since the expansions are done around the same points and in the same order. 
Expanded differently, the same generating function gives the higher cohomology dimensions: 
\begin{equation}\label{highercoh7643}
\left(G(X,t_1,t_2)~,\begin{array}{cc} t_2& t_1 \\0&\infty \end{array}\right)+\left(G(X,t_1,t_2)~,\begin{array}{cc} t_2& t_1 \\\infty&0 \end{array}\right) = CS^1(X,t_1,t_2)-CS^2(X,t_1,t_2)~.
\end{equation}
Since in this case there are no line bundles with non-vanishing first and second cohomology, the two cohomology series can be disentangled and correspond to the positive and, respectively, the negative coefficients in Eq.~\eqref{highercoh7643}.

\end{example}

\subsection{Examples of non-Mori dream spaces}
The examples presented in this section correspond to Calabi-Yau threefolds which admit an infinite sequence of isomorphic flops. 

\subsubsection{Examples in Picard number 2}
\begin{example} 
Let $X$ be the general complete intersection Calabi-Yau threefold belonging to the deformation family
\begin{equation}\label{conf7644}
\cicy{\IP^4 \\ \IP^4}{2& 0& 1 & 1& 1\\ 0& 2& 1& 1& 1}
\end{equation}
with identifier $7644$ in the CICY list and Hodge numbers $(h^{1,1}(X),h^{1,2}(X))=(2,46)$. The effective cone decomposes into a doubly infinite sequence of Mori chambers corresponding to the nef cones of isomorphic Calabi-Yau threefolds connected to $X$ through a sequence of flops. With $H_1 = \mathcal O_{\IP^4\times \IP^4}(1,0)|_X$ and $H_2 = \mathcal O_{\IP^4\times \IP^4}(0,1)|_X$, a general term in this sequence of Mori chambers is
\begin{equation} 
K^{(n)} = \IR_{\geq 0} (a_{n+1} H_1-a_{n}H_2) + \IR_{\geq 0} (a_n H_1-a_{n-1}H_2)
\end{equation}
where $a_n$ is given by 
\begin{equation} 
a_n = \frac{\left(3+2 \sqrt{2}\right)^n-\left(3-2 \sqrt{2}\right)^n}{4 \sqrt{2}}~, ~~(a_n)=\ldots -204,-35,-6,-1,0,1,6,35,204,\ldots
\end{equation}
and $K^{(0)}={\rm Nef}(X)$. The structure of the effective cone is schematically represented in \fref{fig:extendedKC}.

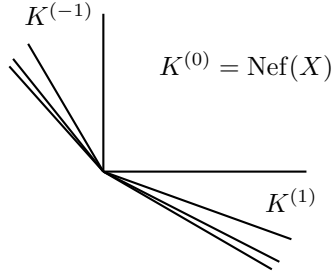
\begin{figure}[H]
\begin{center}
\begin{tikzpicture}[baseline={([yshift=-0.2ex]current bounding box.center)}]
\draw[-,thick] (0,0) -- (0,2.1);
\draw[-,thick] (0,0) -- (2.7,0);
\draw[-,thick] (0,0) -- (-1.,1.7);
\draw[-,thick] (0,0) -- (-1.2,1.5);
\draw[-,thick] (0,0) -- (-1.25,1.4);
\draw[-,thick] (0,0) -- (2.5,-.9);
\draw[-,thick] (0,0) -- (2.34,-1.2);
\draw[-,thick] (0,0) -- (2.24,-1.3);
\node at (1.9,1.4) {$K^{(0)}={\rm Nef}(X)$};
\node at (-.59,2.1) {$K^{(-1)}$};
\node at (2.5,-.38) {$K^{(1)}$};
\end{tikzpicture}
\caption{Movable cone for a general complete intersection Calabi-Yau threefold in the deformation family~\eqref{conf7644}, admitting an infinite sequence of flops.}
\label{fig:extendedKC}
\end{center}
\end{figure}

The generating function for cohomology can be written in terms of the Hilbert-Poincar\'e series associated with the homogeneous coordinate ring of $X$,
\begin{equation}
HS(X,t_1,t_2) = \frac{\left(1-t_1^2\right) \left(1-t_2^2\right) \left(1-t_1 t_2\right)^3}{\left(1-t_1\right)^5 \left(1-t_2\right)^5}
\end{equation} 
and the correction term  
\begin{equation}
\frac{\left(1-t^2\right) \left(1-t^3\right)}{\left(1-t\right)^5}~.
\end{equation} 
Concretely, we define the functions
\begin{equation}
\begin{aligned}
G_n(t_1,t_2)= ~& \frac{ \Big(1-(t_1^{a_{n+1}} t_2^{-a_{n}})^2\Big) \Big(1-(t_1^{a_{n}} t_2^{-a_{n-1}})^2\Big)  \left(1-t_1^{a_n+a_{n+1}} t_2^{-a_{n-1}-a_{n}}\right)^3}{ \left(1-t_1^{a_{n+1}} t_2^{-a_{n}}\right)^5\big(1-t_1^{a_n} t_2^{-a_{n-1}}\big)^5} \\[4pt]
C_n(t_1,t_2)=~ & \frac{\Big(1-(t_1^{a_{n}} t_2^{-a_{n-1}})^2\Big)\Big(1-(t_1^{a_{n}} t_2^{-a_{n-1}})^3\Big)}{\big(1-t_1^{a_n} t_2^{-a_{n-1}}\big)^5}~,
\end{aligned}
\end{equation}
in terms of which the zeroth line bundle cohomology series on $X$ takes the form: 
\begin{equation}
\begin{aligned}
CS^0(X,\mathcal O_X)&  =  \left(\sum_{n=-\infty}^{0} G_n(t_1,t_2){+}C_n(t_1,t_2)\,,\!\begin{array}{cc} t_2& t_1 \\0&0 \end{array}\!\!\right)+\left(\sum_{n=1}^{\infty} G_n(t_1,t_2){+}C_n(t_1,t_2)\,,\!\begin{array}{cc} t_1& t_2 \\0&0 \end{array}\!\!\right)~. 
\end{aligned}
\end{equation}

The correction term is such that, e.g.,~along the wall separating $K^{(-1)}$ from $K^{(0)}$ 
\begin{equation}
\begin{aligned}
\left.\frac{(1-t_1^2) (1-t_2^2) (1-t_1 t_2)^3}{(1-t_1)^5 (1-t_2)^5}\right|_{t_1= 0}& \!\!+~ \left.\frac{(1-(t_1^{-1}t_2^6)^2) (1-t_2^2) (1-(t_1^{-1}t_2^6) t_2)^3}{(1-(t_1^{-1}t_2^6))^5 (1-t_2)^5}\right|_{t_1= \infty } - \frac{(1-t_2^2) (1-t_2^3)}{(1-t_2)^5}
\end{aligned}
\end{equation}
corresponds to the Hilbert-Poincar\'e series of the singular manifold involved in the corresponding flop, which can be represented as the intersection of a sextic and a quadratic polynomial in the coordinates of the weighted projective space $\IP_{[3:1:1:1:1:1]}$, namely
\begin{equation}
HS(\IP_{[3:1:1:1:1:1]}[6~2], t) = \frac{(1 - t^6) (1 - t^2)}{(1 - t^3) (1 - t)^5}~.
\end{equation}
\end{example}

For the higher cohomologies we propose the following. 

\begin{con}\label{con:inf_flops} {\rm(Conjecture~7)}
Let $X$ be a general complete intersection Calabi-Yau threefold in the deformation family given by the configuration matrix 
\begin{equation*}
\cicy{\IP^4 \\ \IP^4}{2& 0& 1 & 1& 1\\ 0& 2& 1& 1& 1}
\end{equation*}
and let $H_1 = \mathcal O_{\IP^4\times \IP^4}(1,0)|_X$ and $H_2 = \mathcal O_{\IP^4\times \IP^4}(0,1)|_X$.
The effective cone decomposes into a doubly infinite sequence of Mori chambers corresponding to the nef cones of isomorphic Calabi-Yau threefolds connected to $X$ through a sequence of flops, of the form $K^{(n)} = \IR_{\geq 0} (a_{n+1} H_1-a_{n}H_2) + \IR_{\geq 0} (a_n H_1-a_{n-1}H_2)$, 
with $a_n$ given by 
\begin{equation} 
a_n = \frac{\left(3+2 \sqrt{2}\right)^n-\left(3-2 \sqrt{2}\right)^n}{4 \sqrt{2}}~, ~~(a_n)=\ldots -204,-35,-6,-1,0,1,6,35,204,\ldots
\end{equation}
such that $K^{(0)}={\rm Nef}(X)$. A generating function for all line bundle cohomology dimensions can be written in terms of the functions
\begin{equation}
\begin{aligned}
G_n(t_1,t_2)= ~& \frac{ (1-(t_1^{a_{n+1}} t_2^{-a_{n}})^2) (1-(t_1^{a_{n}} t_2^{-a_{n-1}})^2)  (1-t_1^{a_n+a_{n+1}} t_2^{-a_{n-1}-a_{n}})^3}{ (1-t_1^{a_{n+1}} t_2^{-a_{n}})^5(1-t_1^{a_n} t_2^{-a_{n-1}})^5} \\[4pt]
C_n(t_1,t_2)=~ & \frac{(1-(t_1^{a_{n}} t_2^{-a_{n-1}})^2)(1-(t_1^{a_{n}} t_2^{-a_{n-1}})^3)}{(1-t_1^{a_n} t_2^{-a_{n-1}})^5}~,
\end{aligned}
\end{equation}
as follows
\begin{equation}
\begin{aligned}
CS^0(X,\mathcal O_X)&  =  \left(\sum_{n=-\infty}^{0} G_n(t_1,t_2){+}C_n(t_1,t_2)\,,\!\begin{array}{cc} t_2& t_1 \\0&0 \end{array}\!\!\right)+\left(\sum_{n=1}^{\infty} G_n(t_1,t_2){+}C_n(t_1,t_2)\,,\!\begin{array}{cc} t_1& t_2 \\0&0 \end{array}\!\!\right) \\
CS^1(X,\mathcal O_X)&  =  \left(\sum_{n=-\infty}^{0} G_n(t_1,t_2){+}C_n(t_1,t_2)\,,\!\begin{array}{cc} t_2& t_1 \\0&\infty \end{array}\!\!\right) /.\,\{\text{remove terms } t_1^\alpha t_2^\beta \text{ with }\alpha+\beta<0 \} ~+ \\
&~~~~~~~~ \left(\sum_{n=0}^{\infty} G_n(t_1,t_2){+}C_n(t_1,t_2)\,,\!\begin{array}{cc} t_1& t_2 \\0&\infty \end{array}\!\!\right) /.\,\{\text{remove terms } t_1^\alpha t_2^\beta \text{ with }\alpha+\beta<0 \}  \\
1-CS^2(X,\mathcal O_X)&  = \left(\sum_{n=-\infty}^{0} G_n(t_1,t_2){+}C_n(t_1,t_2)\,,\!\begin{array}{cc} t_2& t_1 \\ \infty&0 \end{array}\!\!\right) /.\,\{\text{remove terms } t_1^\alpha t_2^\beta \text{ with }\alpha+\beta>0 \}~+\\
&~~~~~~~~ \left(\sum_{n=1}^{\infty} G_n(t_1,t_2){+}C_n(t_1,t_2)\,,\!\begin{array}{cc} t_1& t_2 \\ \infty&0 \end{array}\!\!\right)  /.\,\{\text{remove terms } t_1^\alpha t_2^\beta \text{ with }\alpha+\beta>0 \} \\
1-CS^3(X,\mathcal O_X)&  =  \left(\sum_{n=-\infty}^{0} G_n(t_1,t_2){+}C_n(t_1,t_2)\,,\!\begin{array}{cc} t_2& t_1 \\ \infty& \infty \end{array}\!\!\right)+\left(\sum_{n=0}^{\infty} G_n(t_1,t_2){+}C_n(t_1,t_2)\,,\!\begin{array}{cc} t_1& t_2 \\ \infty&\infty \end{array}\!\!\right)~. \\
\end{aligned}
\end{equation}
\end{con}

The following example is very similar and involves an infinite sequence of isomorphic flops. The difference, however, will be that the correction terms required at the two boundaries of each Mori chamber are different. 

\begin{example} 
Let $X$ be a general CICY threefold belonging to the family with configuration matrix 
\begin{equation}\label{conf7726}
\cicy{\IP^3 \\ \IP^5}{0& 1& 1 & 1& 1\\ 2& 1& 1& 1& 1}
\end{equation}
and identifier $7726$ in the CICY list. The non-trivial Hodge numbers are $(h^{1,1}(X), h^{1,2}(X))=(2,50)$. As in the previous example, the effective cone consists of a doubly infinite sequence of Mori chambers corresponding to the nef cones of isomorphic Calabi-Yau threefolds connected to $X$ through a sequence of flops. With $H_1 = \mathcal O_{\IP^3\times \IP^5}(1,0)|_X$ and $H_2 = \mathcal O_{\IP^3\times \IP^5}(0,1)|_X$, the movable cone consists of the gluing of Mori chambers of the type $\ldots K^{(n-1)}, L^{(n)}, K^{(n)},L^{(n+1)}\ldots$, with
\begin{equation} 
\begin{aligned}
K^{(n)} &= \IR_{\geq 0} (a_{n+1} H_1-b_{n}H_2)+\IR_{\geq 0} (c_{n} H_1-a_{n}H_2) ~,\\
L^{(n)} &= \IR_{\geq 0} (c_{n} H_1-a_{n}H_2)+ \IR_{\geq 0} (a_{n} H_1-b_{n-1}H_2) ~.
\end{aligned}
\end{equation}
The $a_n, b_n$ and $c_n$ integers given by the expressions
\begin{equation} 
\begin{aligned}
a_n &= \left(\frac{\sqrt{30}}{10} - \frac{1}{2}\right)(11 + 2\sqrt{30})^n - \left(\frac{\sqrt{30}}{10} + \frac{1}{2}\right)  (11 - 2\sqrt{30})^n ~,~~(a_n)=\ldots -505, -23, -1, 1, 23\ldots\\
b_n &= \frac{1}{4}\frac{\sqrt{30}}{10}\left((11 + 2\sqrt{30})^n - (11 - 2\sqrt{30})^n\right) ~,~~(b_n)=\ldots -66, -3, 0, 3, 66\ldots\\
c_n &= \frac{2}{3}\frac{\sqrt{30}}{10}\left((11 + 2\sqrt{30})^n - (11 - 2\sqrt{30})^n\right) ~,~~(c_n)=\ldots -176, -8, 0, 8, 176\ldots\\
\end{aligned}
\end{equation}
such that $K^{(0)}={\rm Nef}(X)$. 
The generating function for the zeroth line bundle cohomology can be written in terms of the Hilbert-Poincar\'e series of $X$,
\begin{equation}
HS(X,t_1,t_2) = \frac{\left(1-t_2^2\right)  \left(1-t_1 t_2\right)^4}{\left(1-t_1\right)^4 \left(1-t_2\right)^6}
\end{equation} 
and the correction terms
\begin{equation}
C(t) = \frac{\left(1-t^4\right)}{\left(1-t\right)^4}~,\qquad D(t) = \frac{\left(1-t^2\right)\left(1+t^4\right)}{\left(1-t\right)^6}~.
\end{equation} 
The cohomology series can then be written as:
\begin{equation}
\begin{aligned}
CS^0(X,\mathcal O_X)&  =\\
&\!\!\!\!\!\!\!\!\!\!\!\!\!\!\!\! \!\!\!\!\!\!\!  \left(\sum_{n=-\infty}^{0} HS(X,t_1^{a_{n+1}} t_2^{-b_{n}},t_1^{c_{n}} t_2^{-a_{n}}){+}HS(X, t_1^{c_{n}} t_2^{-a_{n}},t_1^{a_{n}} t_2^{-b_{n-1}}){+}C(t_1^{c_{n}} t_2^{-a_{n}}){+}D(t_1^{a_{n}} t_2^{-b_{n-1}}),\!\!\!\begin{array}{cc} t_2& \!\!\!t_1 \\0&\!\!\!0 \end{array}\!\!\right){+}\\
&\!\!\!\!\!\!\!\!\!\!\!\!\!\!\!\! \!\!\!\!\!\!\!  \left(\sum_{n=1}^{\infty} HS(X,t_1^{a_{n+1}} t_2^{-b_{n}},t_1^{c_{n}} t_2^{-a_{n}}){+}HS(X, t_1^{c_{n}} t_2^{-a_{n}},t_1^{a_{n}} t_2^{-b_{n-1}}){+}C(t_1^{c_{n}} t_2^{-a_{n}}){+}D(t_1^{a_{n}} t_2^{-b_{n-1}}),\!\!\!\begin{array}{cc} t_1& \!\!\!t_2 \\0&\!\!\!0 \end{array}\!\!\right)~.
\end{aligned}
\end{equation}

The singular threefolds associated with the walls separating the Mori chambers $L^{(n)}$ and $K^{(n)}$ belong, for all $n$, to the deformation family $\IP^5[2~4]$. On the other hand, the singular threefolds associated with the walls separating the Mori chambers $K^{(n)}$ and $L^{(n+1)}$ belong, for all $n$ to another deformation family, $\IP_{[4:1:1:1:1:1:1]}[8]$. As discussed in \cite{{Brodie:2021toe}}, while both rows in the configuration matrix~\eqref{conf7726} lead, for general complete intersections in this family, to isomorphic flops, the details of the two flops are quite different. 
\end{example}

\subsubsection{An example in Picard number $4$}
\begin{example}
Consider a general Calabi-Yau hypersurface $X$ of multi-degree $(2,2,2,2)$ in the product of projective spaces $\IP^1\times\IP^1\times\IP^1\times\IP^1$. Line bundle cohomology formulae on such hypersurfaces have been given in Refs.~\cite{Constantin:2022jyd, Abel:2023zwg} and earlier in Refs.~\cite{Constantin:2018otr, Buchbinder:2013dna}. We define $H_1 {=} \mathcal O_{\IP^1{\times}\IP^1{\times} \IP^1{\times}\IP^1}(1,0,0,0)|_X$, $H_2 {=} \mathcal O_{\IP^1{\times}\IP^1{\times}\IP^1{\times}\IP^1}(0,1,0,0)|_X$, $H_3 = \mathcal O_{\IP^1{\times} \IP^1{\times}\IP^1{\times}\IP^1}(0,0,1,0)|_X$, and $H_4 = \mathcal O_{\IP^1\times \IP^1\times \IP^1\times \IP^1}(0,0,0,1)|_X$. The nef cone of $X$ is inherited from the ambient variety
\begin{equation}
{\rm Nef}(X) = \IR_{\geq 0} H_1 +\IR_{\geq 0} H_2 +\IR_{\geq 0} H_3 +\IR_{\geq 0} H_4~.
\end{equation}
The effective cone consists of an infinite number of Mori chambers, whose generators written in the basis $B=\{H_1,H_2,H_3,H_4\}$ are obtained by the action of an infinite group $G=\langle M_1,M_2,M_3,M_4 \rangle$ on $B$. The generators of $G$ correspond to the following matrices
\begin{equation}
\begin{aligned}
&{\scriptsize M_1 {=} \left(\begin{array}{rrrr}{\!\!\!\!-1}&{0}&{0}&{0}\!\!\\{2}&{1}&{0}&{0}\!\!\\{2}&{0}&{1}&{0}\!\!\\{2}&{0}&{0}&{1}\!\!\end{array}\right)~,~
M_2 {=} \left(\begin{array}{rrrr}{\!\!1}&{2}&{0}&{0}\!\!\\{\!\!0}&{\!\!\!\!{-}1}&{0}&{0}\!\!\\{\!\!0}&{2}&{1}&{0}\!\!\\{\!\!0}&{2}&{0}&{1}\!\!\end{array}\right)~,~}
{\scriptsize M_3 {=} \left(\begin{array}{rrrr}{\!\!1}&{0}&{2}&{0}\!\!\\{\!\!0}&{1}&{2}&{0}\!\!\\{\!\!0}&{0}&{\!\!\!\!{-}1}&{0}\!\!\\{\!\!0}&{0}&{2}&{1}\!\!\end{array}\right)~,~~
M_4 {=} \left(\begin{array}{rrrr}{\!\!1}&{0}&{0}&{2}\!\!\\{\!\!0}&{1}&{0}&{2}\!\!\\{\!\!0}&{0}&{1}&{2}\!\!\\{\!\!0}&{0}&{0}&{\!\!\!\!{-}1}\!\!\end{array}\right)~.}
\end{aligned}
\end{equation}

The generating function for zeroth line bundle cohomology can be written in terms of the Hilbert-Poincar\'e series of $X$, with four counting variables
\begin{equation}
HS(X,t_1,t_2,t_3,t_4) = \frac{1 - t_1^2 t_2^2 t_3^2 t_4^2}{(1 - t_1)^2 (1 - t_2)^2 (1 - t_3)^2 (1 - t_4)^2}
\end{equation} 
and the correction term
\begin{equation}
C(x,y,z) = \frac{(1 - x^2 y^2 z^2)}{(1 - x)^2 (1 - y)^2 (1 - z)^2}~.
\end{equation} 

The zeroth line bundle cohomology is then obtained by summing up the contributions from all the Mori chambers and subtracting correction terms associated with every wall separating two Mori chambers, along the same lines as in the previous two examples.   
\end{example}

\subsection{General hypersurfaces of bi-degree $(3,3)$ in $\IP^2\times \IP^2$}
Let $X$ be a general hypersurface of bi-degree $(3,3)$ in $\IP^2\times\IP^2$, correspondiing to a smooth Calabi-Yau threefold with Picard number 2. We discuss the zeroth and the first cohomology of line bundles over $X$; the second and third cohomologies are then related to these by Serre duality. 

Let $H_1 = \mathcal O_{\IP^2\times \IP^2}(1,0)|_X$ and $H_2 = \mathcal O_{\IP^2\times \IP^2}(0,1)|_X$. The effective, movable and nef cones coincide 
\begin{equation}
{\rm Eff}(X)={\rm Mov}(X)={\rm Nef}(X)=\RR_{\ge 0}H_1+\RR_{\ge 0}H_2~.
\end{equation}
The zeroth cohomology function discussed, e.g.,~in Ref.~\cite{Constantin:2018hvl} is given by:
\begin{equation}
h^0(X,L=\cO_X(k_1H_1+k_2H_2)) = 
\begin{cases}
1,~\text{if } k_1=k_2=0\\
\chi(X,L),~\text{if } k_1,k_2>0\\
\chi(\IP^2,\cO_{\IP^2}(k_1)),~\text{if } k_1>0,k_2=0\\
\chi(\IP^2,\cO_{\IP^2}(k_2)),~\text{if } k_1=0,k_2>0
\end{cases}
\end{equation}

This information can be compactly encoded in the Hilbert-Poincar\'e series: 
\begin{equation}
CS^0(X,\cO_X)=HS(X,t_1,t_2) =  \left( \frac{1-t_1^3 t_2^3}{(1-t_1)^3(1-t_1)^3} \,,\!\begin{array}{cc} t_1& t_2 \\0&0 \end{array}\!\!\right) = \sum_{m_1,m_2\in\mathbb Z} h^0(X,m_1H_1+m_2H_2) \, t_1^{m_1}t_2^{m_2}~.
\end{equation}

The first line bundle cohomology is non-trivial within the two disconnected cones $\IR_{>0}H_2+\IR_{\geq 0}(-H_1+H_2)$ and $\IR_{>0} H_1+\IR_{\geq 0}(H_1-H_2)$ (\fref{bicubic}).

\begin{figure}[htbp]
\begin{center}
     \includegraphics[width=7.1cm]{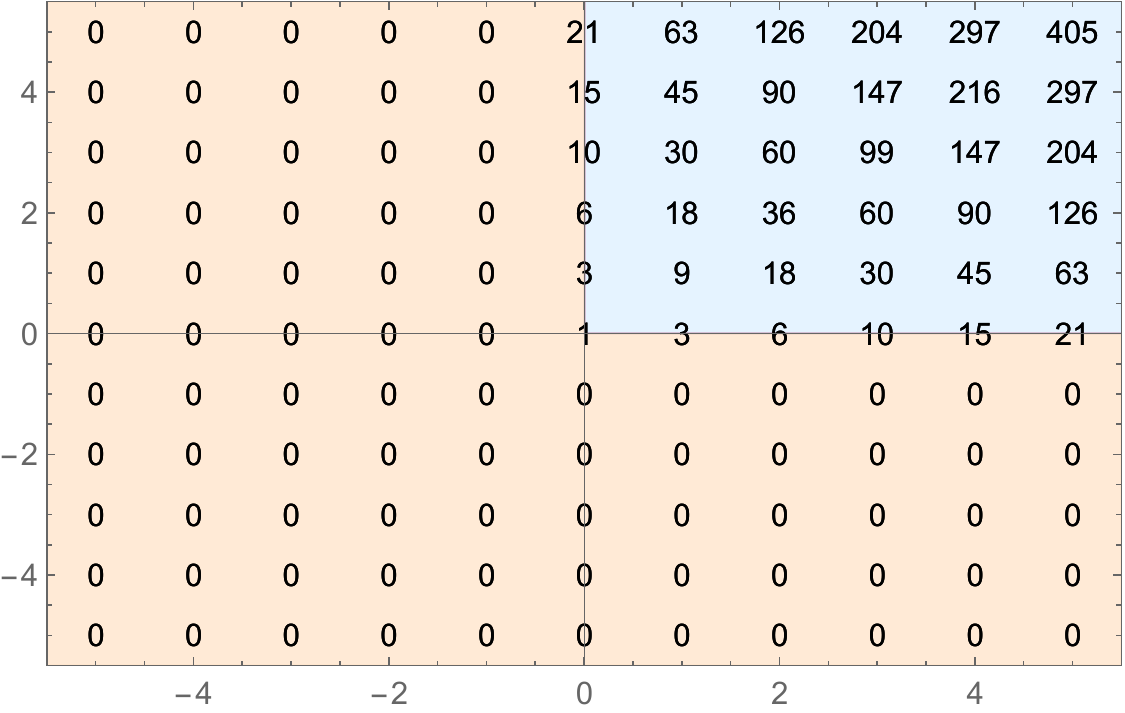}~~~~~
          \includegraphics[width=7.1cm]{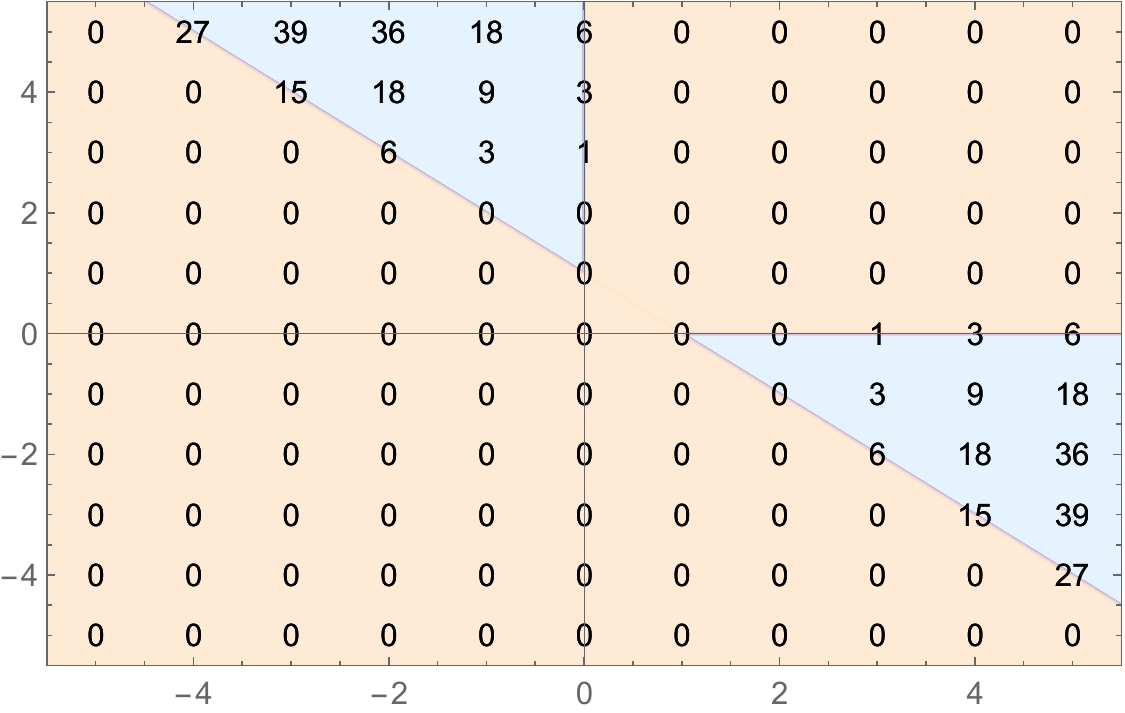}
\caption{\itshape Zeroth and first line bundle cohomology data (left plot and, respectively, right plot) for a general Calabi-Yau threefold hypersurface of bi-degree $(3,3)$ in $\IP^2\times \IP^2$. The numbers indicate cohomology dimensions; the location in the plot indicates the first Chern class.}\label{bicubic}
\end{center}
\end{figure}

\begin{con}\label{con:bicubic} ${\rm (Conjecture~6)}$
Let $X$ be a general hypersurface of bidegree $(3,3)$ in $\IP^2\times \IP^{2}$. Let $H_1 = \mathcal O_{\IP^2\times \IP^2}(1,0)|_X$ and $H_2 = \mathcal O_{\IP^2\times \IP^2}(0,1)|_X$. 
Defining
\begin{equation}
G(x,y) =
\frac{(x^{-1} y)^3 ((1 + x - y)^3 - 1 + 3 x (1 - y))}{(1 - x^{-1} y)^3 (1 - y)^3}~,
\end{equation}
all line bundle cohomology dimensions on $X$ are encoded in the following generating functions, written in the basis $\{H_1,H_2\}$:
\begin{equation}
\begin{aligned}
CS^0(X,\mathcal O_X)&  = 1 +  \left( G(t_1,t_2)\,,\!\begin{array}{cc} t_1& t_2 \\0&0 \end{array}\!\!\right)+\left( G(t_2,t_1)\,,\!\begin{array}{cc} t_1& t_2 \\0&0 \end{array}\!\!\right) \\
CS^1(X,\mathcal O_X)& = 0 +  \left( G(t_1,t_2)\,,\!\begin{array}{cc} t_1& t_2 \\ \infty&0 \end{array}\!\!\right)+\left( G(t_2,t_1)\,,\!\begin{array}{cc} t_1& t_2 \\0&\infty \end{array}\!\!\right) \\
-CS^2(X,\mathcal O_X)& = 2+ \left( G(t_1,t_2)\,,\!\begin{array}{cc} t_1& t_2 \\ 0&\infty \end{array}\!\!\right)+\left( G(t_2,t_1)\,,\!\begin{array}{cc} t_1& t_2 \\ \infty& 0 \end{array}\!\!\right) \\
-CS^3(X,\mathcal O_X)& =  1+ \left( G(t_1,t_2)\,,\!\begin{array}{cc} t_1& t_2 \\ \infty&\infty \end{array}\!\!\right)+\left( G(t_2,t_1)\,,\!\begin{array}{cc} t_1& t_2 \\ \infty&\infty \end{array}\!\!\right) 
\end{aligned}
\end{equation}
\end{con}

\section{Conclusion and Outlook}
While sheaf cohomology is a powerful tool in algebraic geometry and has direct implications for key quantities in mathematical physics, the current methods in computational algebraic geometry are limited in their applicability. 
 Although in simple cases it is possible to compute sheaf cohomology explicitly from the properties apparent from its definition, it can be hard to obtain these groups for more involved examples such as the manifolds routinely employed in string compactifications. In this paper we proposed exact generating functions for line bundle cohomology dimensions on several complex projective varieties of dimension 2 and higher, of Fano, Calabi-Yau and general type and low Picard number (mostly in Picard number 2). The existence of such generating functions is akin to the relation between the Hilbert series and the Euler characteristic of line bundles. The evidence gathered here suggests that a single generating function can encode not only the zeroth cohomology, but also the higher cohomologies of all line bundles on the variety, which prompts the question of whether these generating functions uniquely determine the variety.

\section*{Acknowledgements}
I am grateful to Callum Brodie for early collaboration on this project and for many discussions on closed-form cohomology formulae. I thank Aurelio Carlucci, James Gray, Andre Lukas, John Christian Ottem, Fabian Ruehele, Hal Schenck and Michael Stillman for discussions. This research was supported by a Stephen Hawking Fellowship, EPSRC grant EP/T016280/1, and by a Royal Society Dorothy Hodgkin Fellowship.

\bibliographystyle{utcaps}

\providecommand{\href}[2]{#2}\begingroup\raggedright\endgroup

\end{document}